\documentclass[11pt,a4paper]{article}
\usepackage{amsmath}\usepackage{epsf,amsfonts,amsthm}\usepackage{amscd,amssymb}
\usepackage{xcolor,epic,eepic}\usepackage{epsfig}
\usepackage{fontenc,indentfirst, delarray,amsfonts,amsmath,amssymb}
\usepackage{rotating}
\usepackage{mathdots}
\usepackage[T1]{fontenc}
\usepackage[matrix,arrow,curve]{xy}
\usepackage{amsmath}
\usepackage{amssymb}
\usepackage{amsthm}
\usepackage{amscd}
\usepackage{amsfonts}
\usepackage{graphicx}
\usepackage{fancyhdr}
\usepackage{dsfont,texdraw}
\usepackage{amsmath}\usepackage{epsf,amsfonts,amsthm}\usepackage{amscd,amssymb}
\usepackage{xcolor,epic,eepic}\usepackage{epsfig}
\usepackage{fontenc,indentfirst, delarray,amsfonts,amsmath,amssymb}
\usepackage{rotating}
\usepackage{amscd}
\usepackage{amsmath}
\usepackage{amsthm}
\usepackage{mathrsfs}
\usepackage{latexsym}
\usepackage{amssymb}
\usepackage{amsfonts}
\theoremstyle{plain}

\usepackage{tikz}
\usepackage{tikz-cd}
\usetikzlibrary{arrows,chains,matrix,positioning,scopes,snakes}
\makeatletter
\tikzset{join/.code=\tikzset{after node path={%
\ifx\tikzchainprevious\pgfutil@empty\else(\tikzchainprevious)%
edge[every join]#1(\tikzchaincurrent)\fi}}}
\makeatother

\tikzset{>=stealth',every on chain/.append style={join},
         every join/.style={->}}
\tikzset{
    >=stealth',
    punkt/.style={
           rectangle,
           rounded corners,
           draw=black, very thick,
           text width=6.5em,
           minimum height=2em,
           text centered},
    pil/.style={
           ->,
           thick,
           shorten <=2pt,
           shorten >=2pt,}
}

\usepackage[T1]{fontenc}
\newcommand{\bee}{\begin{enumerate}}
\newcommand{\eee}{\end{enumerate}}
\newcommand{\benn}{\begin{equation*}}
\newcommand{\eenn}{\end{equation*}}
\newcommand{\be}{\begin{equation}}
\newcommand{\ee}{\end{equation}}
\newcommand{\bean}{\begin{eqnarray}}
\newcommand{\eean}{\end{eqnarray}}
\newcommand{\bea}{\begin{eqnarray*}}
\newcommand{\eea}{\end{eqnarray*}}

\newcommand{\p}{\partial}

\newcommand{\ra}{\rangle}

\newcommand{\Ci}{C^{\infty}}

\newcommand{\N}{\mathbb{N}}

\newcommand{\R}{\mathbb{R}}
\newcommand{\K}{\mathbb{K}}

\newcommand{\Q}{\mathbb{Q}}

\newcommand{\lp}{\left(}
\newcommand{\rp}{\right)}

\newcommand{\op}[1]{\!\!\mathop{\rm ~#1}\nolimits}

\newcommand{\mbi}{\mathbb{I}}

\mathchardef\za="710B  
\mathchardef\zb="710C  
\mathchardef\zg="710D  
\mathchardef\zd="710E  
\mathchardef\zve="710F 
\mathchardef\zz="7110  
\mathchardef\zh="7111  
\mathchardef\zy="7112 

\mathchardef\zi="7113  
\mathchardef\zk="7114  
\mathchardef\zl="7115  
\mathchardef\zm="7116  
\mathchardef\zn="7117  
\mathchardef\zx="7118  
\mathchardef\zp="7119  
\mathchardef\zr="711A  
\mathchardef\zs="711B  
\mathchardef\zt="711C  
\mathchardef\zu="711D  
\mathchardef\zf="711E 
\mathchardef\zq="711F  
\mathchardef\zc="7120  
\mathchardef\zw="7121  
\mathchardef\ze="7122  
\mathchardef\zvy="7123  
\mathchardef\zvw="7124  
\mathchardef\zvr="7125 
\mathchardef\zvs="7126 
\mathchardef\zvf="7127  
\mathchardef\zG="7000  
\mathchardef\zD="7001  
\mathchardef\zY="7002  
\mathchardef\zL="7003  
\mathchardef\zX="7004  
\mathchardef\zP="7005  
\mathchardef\zS="7006  
\mathchardef\zU="7007  
\mathchardef\zF="7008  
\mathchardef\zW="700A  

\newcommand{\cyclic}{\mathop{\kern0.9ex{{+}
\kern-2.15ex\raise-.25ex\hbox{\Large\hbox{$\circlearrowright$}}}}\limits}

\newcommand{\cE}{{\cal E}}
 \newcommand{\cS}{{\cal S}}

 \newcommand{\cH}{{\cal H}}
 \newcommand{\cP}{{\cal P}}
 \newcommand{\cC}{{\cal C}}
 \newcommand{\cA}{{\cal A}}
 \newcommand{\cM}{{\cal M}}
 \newcommand{\cD}{{\cal D}}
 \newcommand{\cO}{{\cal O}}
 \newcommand{\cB}{{\cal B}}
 
 \newcommand{\cQ}{{\cal Q}}
 \newcommand{\cI}{{\cal I}}

\newtheorem{rem}{Remark}
\newtheorem{theo}[rem]{Theorem}
\newtheorem{prop}[rem]{Proposition}
\newtheorem{lem}[rem]{Lemma}
\newtheorem{cor}[rem]{Corollary}
\newtheorem{ex}[rem]{Example}

\newtheorem{defi}[rem]{Definition}

\newcommand{\h}{\op{Hom}}
\newcommand{\0}{\otimes}

\newcommand{\id}{\op{id}}
\newcommand{\coker}{\op{coker}}
\newcommand{\im}{\op{im}}

\DeclareMathAlphabet{\mathpzc}{OT1}{pzc}{m}{it}

 \newcommand{\cJ}{\mathcal{J}}

 \newcommand{\colim}{\op{colim}}

\begin{document}
\title{\bf Koszul-Tate resolutions as cofibrant replacements of algebras over differential operators}
\date{}
\author{Gennaro di Brino, Damjan Pi\v{s}talo, and Norbert Poncin\footnote{University of Luxembourg, Mathematics Research Unit, 4364 Esch-sur-Alzette, Luxembourg, gennaro.dibrino@gmail.com, damjan.pistalo@uni.lu, norbert.poncin@uni.lu}}

\maketitle

\begin{abstract} Homotopical geometry over differential operators is a convenient setting for a coordinate-free investigation of nonlinear partial differential equations modulo symmetries. One of the first issues one meets in the functor of points approach to homotopical $\cD$-geometry, is the question of a model structure on the category $\tt DGAlg(\cD)$ of differential non-negatively graded $\cO$-quasi-coherent sheaves of commutative algebras over the sheaf $\cD$ of differential operators of an appropriate underlying variety $(X,\cO)$. We define a cofibrantly generated model structure on $\tt DGAlg(\cD)$ via the definition of its weak equivalences and its fibrations, characterize the class of cofibrations, and build an explicit functorial `cofibration - trivial fibration' factorization. We then use the latter to get a functorial model categorical Koszul-Tate resolution for $\cD$-algebraic `on-shell function' algebras (which contains the classical Koszul-Tate resolution). The paper is also the starting point for a homotopical $\cD$-geometric Batalin-Vilkovisky formalism.
\end{abstract}

\vspace{2mm} \noindent {\bf MSC 2010}: 18G55, 16E45, 35A27, 32C38, 16S32, 18G10\smallskip

\noindent{\bf Keywords}: Differential operator, $\cD$-module, model category, relative Sullivan $\cD$-algebra, homotopical geometry, $\cD$-geometry, functor of points, Koszul-Tate resolution, Batalin-Vilkovisky formalism
\thispagestyle{empty}

{\tableofcontents}

\section{Introduction}

The solution functor of a {\it linear} {\small PDE} $D\cdot m=0$ is a functor $\op{Sol}:{\tt Mod}(\cal D)\to {\tt Set}$ defined on the category of left modules over the ring $\cD$ of linear differential operators of a suitable underlying space: for $D\in\cD$ and $M\in{\tt Mod}(\cD)$, we have $$\op{Sol}(M)=\{m\in M: D\cdot m=0\}\;.$$ For a {\it polynomial} {\small PDE}, we get a representable functor $\op{Sol}:{\tt Alg}(\cD)\to {\tt Set}$ defined on the category of $\cD$-algebras, i.e., of commutative monoids in ${\tt Mod}(\cD)$. According to \cite{BD}, the solution functor of a {\it nonlinear} {\small PDE} should be viewed as a `locally representable' sheaf $\op{Sol}:{\tt Alg}({\cal D})\to {\tt Set}$. To allow for still more general spaces, sheaves ${\tt Alg}({\cal D})\to {\tt SSet}$ valued in simplicial sets, or sheaves ${\tt DGAlg}({\cal D})\to {\tt SSet}$ on (the opposite of) the category ${\tt DGAlg}({\cal D})$ of differential graded $\cal D$-algebras, have to be considered.\medskip

More precisely, {when constructing a derived algebraic variant of the jet bundle approach to the Lagrangian Batalin-Vilkovisky formalism}, not, as usual, in the world of function algebras, but, dually, on the space side, we first consider the quotient of the infinite jet space by the global gauge symmetries. It turns out \cite{BPP3} that this quotient should be thought of as a 1-geometric derived $X$-$\cD_X$-stack, where $X$ is an underlying smooth affine algebraic variety. This new homotopical algebraic $\cD$-geometry provides in particular a convenient way to encode total derivatives and it allows actually to recover the classical Batalin-Vilkovisky complex as a specific case of its general constructions \cite{PP17}. In the functor of points approach to spaces, the derived $X$-$\cD_X$-stacks $F$ are those presheaves $F:{\tt DGAlg}(\cD)\to \tt SSet$ that satisfy the fibrant object (sheaf-)condition for the local model structure on the presheaf category $\tt Fun(DGAlg(\cD),SSet)$ -- the category of derived $X$-$\cD_X$-stacks is in fact the homotopy category of this model category of functors. More precisely, the choice of an adequate model (pre-)topology {allows us} to construct the local model structure, via a double Bousfield localization, from the global model structure of the considered presheaf category, which is implemented `object-wise' by the model structure of the target category $\tt SSet$. The first of the two Bousfield localizations is the localization of this global model structure with respect to the weak equivalences of the (category opposite to the) source category $\tt DGAlg(\cD)$. Furthermore, the $\cD$-geometric counterpart of an algebra $\Ci(\zS)$ of on-shell functions is an algebra $A\in{\tt Alg}(\cD)\subset {\tt DGAlg}(\cD)$, and it appears \cite{PP} that the Koszul-Tate resolution of $\Ci(\zS)$ corresponds to the cofibrant replacement of $A$ in a coslice category of $\tt DGAlg(\cD)$.\medskip

In view of the two preceding reasons, it is clear that our first task is the definition of a model structure on ${\tt DGAlg}(\cD)$ ({we draw the attention of the reader to the fact that we will use two different definitions of model categories, namely the definition of \cite{GS} and that of \cite{Hov} -- for the details we refer to Appendix \ref{ModCat})}. In the present paper, we give an explicit description of a cofibrantly generated model structure on the category ${\tt DGAlg}(\cD)$ of differential non-negatively graded $\cO$-quasi-coherent sheaves of commutative algebras over the sheaf $\cD$ of differential operators of a smooth affine algebraic variety $(X,\cO)$. In particular, we characterize the cofibrations as the retracts of the relative Sullivan $\cD$-algebras and we give an explicit functorial `{$\,$\small Cof -- TrivFib}' factorization (as well as the corresponding functorial cofibrant replacement functor -- which is specific to our setting and is of course different from the one provided, for arbitrary cofibrantly generated model categories, by the small object argument).\medskip 

To develop the afore-mentioned homotopical $\cD$-geometry, we have to show inter alia that the triplet $\tt (DGMod(\cD),DGMod(\cD),DGAlg(\cD))$ is a homotopical algebraic context \cite{TV08}. This includes proving that the model category $\tt DGAlg(\cD)$ is proper and that the base change functor $\cB\0_\cA-\,$, from modules in $\tt DGMod(\cD)$ over $\cA\in\tt DGAlg(\cD)$ to modules over $\cB\in\tt \cA\downarrow DGAlg(\cD)$, preserves weak equivalences. These results \cite{BPP3} are based on our characterization of cofibrations in $\tt DGAlg(\cD)$, as well as on the explicit functorial `{$\,$\small Cof -- TrivFib}' factorization.\medskip 

Notice finally that our two assumptions -- smooth and affine -- on the underlying variety $X$ are necessary. Exactly the same smoothness condition is indeed used in \cite{BD}$\,$[Remark p.56], since for an arbitrary singular scheme $X$, the notion of left $\cD_X$-module is meaningless. On the other hand, the assumption that $X$ is affine is needed to {substitute global sections of sheaves to the sheaves themselves}, i.e., to replace the category of differential non-negatively graded $\cO$-quasi-coherent sheaves of commutative algebras over the sheaf $\cD$ by the category of differential non-negatively graded commutative algebras over the ring $\cD(X)$ of global sections of $\cD$. However, this confinement is not merely a comfort solution: the existence of the projective model structure -- that we transfer from $\tt DGMod(\cD)$ to $\tt DGAlg(\cD)$ -- requires that the underlying category has enough projectives, and this is in general not the case for a category of sheaves over a not necessarily affine scheme \cite{Gil06}, \cite[Ex.III.6.2]{Har}. In addition, the confinement to the affine case allows us to use the artefacts of the model categorical environment, as well as to extract the fundamental structure of the main actors of the considered problem, which may then be extended to an arbitrary smooth scheme $X$ \cite{PP}.\medskip

Let us still stress that the special behavior of the noncommutative ring $\cD$ turns out to be a source of possibilities, as well as a source of problems. For instance, a differential graded commutative algebra over a field or a commutative ring $k$ is a commutative monoid in the category of differential graded $k$-modules. The extension of this concept to noncommutative rings $R$ is problematic, since the category of differential graded (left) $R$-modules is not symmetric monoidal. In the case $R=\cD$, we deal with differential graded (left) $\cD$-modules and these {\it are} symmetric monoidal -- and also closed. However, the tensor product and the internal Hom are taken, not over $\cD$, but over $\cO$: one considers the $\cO$-modules given, for $M,N\in\tt DGMod(\cD)$, by $M\0_\cO N$ and $\h_{\cO}(M,N)$, and shows that their $\cO$-module structures can be extended to $\cD$-module structures {\cite{HTT}}. This and other -- in particular related -- specificities must be kept in mind throughout the whole paper.

\section{Conventions and notation}

According to the anglo-saxon nomenclature, we consider the number 0 as being neither positive, nor negative.\medskip

All the rings used in this text are implicitly assumed to be unital.

\section{Sheaves of modules}\label{ShMod}

\newcommand{\cR}{{\cal R}}
\newcommand{\cF}{{\cal F}}

Let $\tt Top$ be the category of topological spaces and, for $X\in \tt Top$, let ${\tt Open}_X$ be the category of open subsets of $X$. If $\cR_X$ is a sheaf of rings, a {\bf left $\cR_X$-module} is a {\it sheaf $\cP_X$, such that, for each $U\in{\tt Open}_X$, $\cP_X(U)$ is an $\cR_X(U)$-module, and the $\cR_X(U)$-actions are compatible with the restrictions}. We denote by ${\tt Mod}(\cR_X)$ the abelian category of $\cR_X$-modules and of their (naturally defined) morphisms.\medskip

In the following, we omit subscript $X$ if no confusion arises.\medskip

If $\cP,\cQ\in {\tt Mod}(\cR)$, the (internal) Hom {denoted by} ${\cH}om_{\cR}(\cP,\cQ)$ is the sheaf of abelian groups (of $\cR$-modules, i.e., is the element of ${\tt Mod}(\cR)$, if $\cR$ is commutative) that is defined by \be{\cH}om_{\cR}(\cP,\cQ)(U):=\op{Hom}_{\cR|_U}(\cP|_U,\cQ|_U)\;,\label{HomSh}\ee $U\in{\tt Open}_X$. The {\small RHS} is made of the morphisms of (pre)sheaves of $\cR|_U$-modules, i.e., of the families $\zf_V:\cP(V)\to \cQ(V)$, $V\in{\tt Open}_U$, of $\cR(V)$-linear maps that commute with restrictions. Note that ${\cH}om_{\cR}(\cP,\cQ)$ is a sheaf of abelian groups, whereas $\op{Hom}_{\cR}(\cP,\cQ)$ is the abelian group of morphisms of (pre)sheaves of $\cR$-modules. We thus obtain a bi-functor \be\label{HomFun}{\cal H}om_{\cR}(\bullet,\bullet): ({\tt Mod}(\cR))^{\op{op}}\times {\tt Mod}(\cR)\to {\tt Sh}(X)\;,\ee valued in the category of sheaves of abelian groups, which is left exact in both arguments. \medskip

Further, if $\cP\in{\tt Mod}(\cR^{\op{op}})$ and $\cQ\in{\tt Mod}(\cR)$, we denote by $\cP\otimes_\cR\cQ$ the sheaf of abelian groups (of $\cR$-modules, if $\cR$ is commutative) associated to the presheaf \be\label{TensSh}(\cP\oslash_\cR\cQ)(U):=\cP(U)\otimes_{\cR(U)}\cQ(U)\;,\ee $U\in{\tt Open}_X$. The bi-functor \be\label{TensFun}\bullet\0_{\cR}\bullet:{\tt Mod}(\cR^{\op{op}})\times{\tt Mod}(\cR)\to {\tt Sh}(X)\;\ee is right exact in its two arguments.\medskip

If $\cS$ is a sheaf of commutative rings and $\cR$ a sheaf of rings, and if $\cS\to\cR$ is a morphism of sheafs of rings, whose image is contained in the center of $\cR$, we say that $\cR$ is a sheaf of $\cS$-algebras. Remark that, in this case, the above functors ${\cH}om_{\cR}(\bullet,\bullet)$ and $\bullet\0_\cR\bullet$ are valued in ${\tt Mod}(\cS)$.

\section{$\cD$-modules and $\cD$-algebras}

Depending on the author(s), the concept of $\cD$-module is considered over a base space $X$ that is a finite-dimensional smooth \cite{Cos} or complex \cite{KS} manifold, or a smooth algebraic variety \cite{HTT} or scheme \cite{BD}, over a fixed base field $\K$ of characteristic zero. We denote by $\cO_X$ (resp., $\Theta_X$, $\cD_X$) the sheaf of functions (resp., vector fields, differential operators acting on functions) of $X$, and take an interest in the category ${\tt Mod}(\cO_X)$ (resp., ${\tt Mod}(\cD_X)$) of $\cO_X$-modules (resp., $\cD_X$-modules).\medskip

Sometimes a (sheaf of) $\cD_X$-module(s) is systematically required to be {\it coherent} or {\it quasi-coherent} as (sheaf of) $\cO_X$-module(s). In this text, we will explicitly mention such extra assumptions.

\subsection{Construction of $\cD$-modules from $\cO$-modules}\label{D-ModulesAlgebras}

It is worth recalling the following

\begin{prop}\label{DModFlatConnSh} Let $\cM_X$ be an $\cO_X$-module. A left $\,\cD_X$-module structure on $\cM_X$ that extends its $\cO_X$-module structure is equivalent to a $\K$-linear morphism
$$\nabla :  \Theta_X \to {\cal E}nd_\K(\cM_X)\;,$$ such that, for all $f\in\cO_X$, $\zy,\zy'\in\Theta_X$, and all $m\in\cM_X$,
\begin{enumerate}
\item $\nabla_{f\zy}\,m = f\cdot\nabla_\zy m\,,$
\item $\nabla_\zy(f\cdot m)=f\cdot\nabla_\zy m+\zy(f)\cdot m\,,$
\item $\nabla_{[\zy,\zy']}m=[\nabla_\zy,\nabla_{\zy'}]m\,.$
\end{enumerate}
\end{prop}\medskip

In the following, we omit again subscript $X$, whenever possible.\medskip

In Proposition \ref{DModFlatConnSh}, the target ${\cE}nd_\K(\cM)$ is interpreted in the sense of Equation (\ref{HomSh}), and $\nabla$ is viewed as a morphism of sheaves of $\K$-vector spaces. Hence, $\nabla$ is a family $\nabla^U$, $U\in{\tt Open}_X$, of $\K$-linear maps that commute with restrictions, and $\nabla^U_{\zy_U}$, $\zy_U\in\Theta(U)$, is a family $(\nabla^U_{\zy_U})_V$, $V\in{\tt Open}_U$, of $\K$-linear maps that commute with restrictions. It follows that $\lp\nabla^U_{\zy_U}m_U\rp|_V=\nabla^V_{\zy_U|_V}m_U|_V\,$, with self-explaining notation: {\it the concept of sheaf morphism captures the locality of the connection $\nabla$ with respect to both arguments}.\medskip

Further, the requirement that the conditions (1) -- (3) be satisfied for all $f\in\cO$, $\zy,\zy'\in\Theta$, and $m\in\cM$, means that they must hold for any $U\in{\tt Open}_X$ and all $f_U\in\cO(U)$, $\zy_U,\zy_U'\in\Theta(U)$, and $m_U\in\cM(U)$.\medskip

We now detailed notation used in Proposition \ref{DModFlatConnSh}. An explanation of the underlying idea of this proposition can be found in Appendix \ref{D-modules}.

\subsection{Closed symmetric monoidal structure on ${\tt Mod}(\cD)$}

\newcommand{\cL}{{\cal L}}
\newcommand{\cN}{{\cal N}}

If we apply the Hom bi-functor (resp., the tensor product bi-functor) over $\cD$ (see (\ref{HomFun}) (resp., see (\ref{TensFun}))) to two left $\cD$-modules (resp., a right and a left $\cD$-module), we get only a (sheaf of) $\K$-vector space(s) (see remark at the end of Section \ref{ShMod}). {\it The proper concept is the Hom bi-functor (resp., the tensor product bi-functor) over $\cO$}. Indeed, if $\cP,\cQ\in{\tt Mod}(\cD_X)\subset {\tt Mod}(\cO_X)$, the Hom sheaf ${\cal H}om_{\cO_X}(\cP,\cQ)$ (resp., the tensor product sheaf $\cP\0_{\cO_X} \cQ$) is a sheaf of $\cO_X$-modules. To define on this $\cO_X$-module, an extending left $\cD_X$-module structure, it suffices, as easily checked, to define the action of $\theta\in\Theta_X$ on $\zf\in{\cH}om_{\cO_X}(\cP,\cQ)$, for any $p\in\cP$, by \be\label{HomDMod}(\nabla_\theta\zf)(p)=\nabla_\theta(\zf(p))-\zf(\nabla_\theta p)\;\ee ( resp., on $p\0 q$, $p\in\cP, q\in \cQ$, by \be\label{TensDMod}\nabla_\theta(p\0 q)=(\nabla_\theta p)\0 q+p\0(\nabla_\theta q)\;)\;.\ee

The functor $${\cal H}om_{\cO_X}(\cP,\bullet):{\tt Mod}(\cD_X)\to {\tt Mod}(\cD_X)\;,$$ $\cP\in{\tt Mod}(\cD_X)$, is the right adjoint of the functor $$\bullet\0_{\cO_X}\cP:{\tt Mod}(\cD_X)\to {\tt Mod}(\cD_X)\;:$$ for any $\cN,\cP,\cQ\in{\tt Mod}(\cD_X)$, there is an isomorphism $${\cH}om_{\cD_X}(\cN\0_{\cO_X}\cP,\cQ)\ni f \mapsto (n\mapsto(p\mapsto f(n\0 p)))\in{\cH}om_{\cD_X}(\cN,{\cal H}om_{\cO_X}(\cP,\cQ))\;.$$
Hence, {\it the category $({\tt Mod}(\cD_X),\0_{\cO_X},\cO_X,{\cH}om_{\cO_X})$ is abelian closed symmetric monoidal}. More details on $\cD$-modules can be found in \cite{KS, Scha, Schn}.

\begin{rem} In the following, the underlying space $X$ is a smooth algebraic variety over an algebraically closed field $\K$ of characteristic 0.\end{rem}

We denote by ${\tt qcMod}(\cO_X)$ (resp., ${\tt qcMod}(\cD_X)$) the abelian category of quasi-coherent $\cO_X$-modules (resp., $\cD_X$-modules that are quasi-coherent as $\cO_X$-modules \cite{HTT}). This category is a full subcategory of ${\tt Mod}(\cO_X)$ (resp., ${\tt Mod}(\cD_X)$). Since further the tensor product of two quasi-coherent $\cO_X$-modules (resp., $\cO_X$-quasi-coherent $\cD_X$-modules) is again of this type, and since $\cO_X\in{\tt qcMod}(O_X)$ (resp., $\cO_X\in {\tt qcMod}(\cD_X)$), the category $({\tt qcMod}(\cO_X),\0_{\cO_X},\cO_X)$ (resp., $({\tt qcMod}(\cD_X),\0_{\cO_X},\cO_X)$) is a symmetric monoidal subcategory of $({\tt Mod}(\cO_X),\0_{\cO_X},\cO_X)$ (resp., $({\tt Mod}(\cD_X),\0_{\cO_X},\cO_X)$). For additional information on quasi-coherent modules over a ringed space, we refer to Appendix \ref{FinCondShMod}.

\subsection{Commutative $\cD$-algebras}

A $\cD_X$-algebra is a commutative monoid in the symmetric monoidal category ${\tt Mod}(\cD_X)$. More explicitly, a {commutative} $\cD_X$-algebra is a $\cD_X$-module $\cA,$ together with $\cD_X$-linear maps $$\zm:\cA\0_{\cO_X}\cA\to \cA\quad\text{and}\quad \zi:\cO_X\to \cA\;,$$ which respect the usual associativity, unitality, and commutativity conditions. This means exactly that $\cA$ is a commutative associative unital $\cO_X$-algebra, which is endowed with a flat connection $\nabla$ -- see Proposition \ref{DModFlatConnSh} -- such that vector fields $\zy$ act as derivations $\nabla_\zy$. Indeed, when omitting the latter requirement, we forget the linearity of $\zm$ and $\zi$ with respect to the action of vector fields. Let us translate the $\Theta_X$-linearity of $\zm$. If $\zy\in\Theta_X,$ $a,a'\in\cA$, and if $a\ast a':=\zm(a\0 a')$, we get \be \nabla_\theta(a\ast a')=\nabla_\theta(\zm(a\0 a'))=\zm((\nabla_\theta a)\0 a'+a\0 (\nabla_\theta a'))=(\nabla_\theta a)\ast a'+a\ast (\nabla_\theta a')\;.\label{LeibDAlg}\ee If we set now $1_\cA:=\zi(1)$, Equation (\ref{LeibDAlg}) shows that $\nabla_\zy (1_\cA)=0$. It is easily checked that the $\Theta_X$-linearity of $\zi$ does not encode any new information. Hence,

\begin{defi} A commutative {\bf $\cD_X$-algebra} is a commutative monoid in ${\tt Mod}(\cD_X)$, i.e., a commutative associative unital $\cO_X$-algebra that is endowed with a flat connection $\nabla$ such that $\nabla_\zy$, $\zy\in\Theta_X,$ is a derivation.\end{defi}

\begin{rem}{ All $\cD_X$-algebras considered throughout this text are implicitly assumed to be commutative.}\end{rem}

\section{Differential graded $\cD$-modules and differential graded $\cD$-algebras}

\subsection{Monoidal categorical equivalence between chain complexes of $\cD_X$-modules and {their global sections}}\label{MonEquShMod}

It is well known that any equivalence $F:{\tt C} \rightleftarrows {\tt D}:G$ between abelian categories is exact. Moreover, if $F:{\tt C} \rightleftarrows {\tt D}:G$ is an equivalence between monoidal categories, and if one of the functors $F$ or $G$ is strong monoidal, then the other is strong monoidal as well \cite{Hopf}.\medskip

In addition, see (\ref{ShVsSectqcMod}), for any affine algebraic variety $X$, we have the equivalence \be\label{ShVsSectqcMod1}\zG(X,\bullet):{\tt qcMod}(\cO_X)\rightleftarrows {\tt Mod}(\cO_X(X)):\widetilde\bullet\;\ee between abelian symmetric monoidal categories, where $\widetilde\bullet$ is isomorphic to $\cO_X\0_{\cO_X(X)}\bullet\,$. Since the latter is obviously strong monoidal, both functors, $\zG(X,\bullet)$ and $\widetilde\bullet\,$, are exact and strong monoidal. Similarly,

\begin{prop}\label{MonoidEquiv1} If $X$ is a smooth affine algebraic variety, its global section functor $\zG(X,\bullet)$ yields an equivalence \be\label{ShVsSectqcModDMon}\zG(X,\bullet):({\tt qcMod}(\cD_X),\0_{\cO_X},\cO_X)\to ({\tt Mod}(\cD_X(X)),\0_{\cO_X(X)},\cO_X(X))\;\ee between abelian symmetric monoidal categories, and it is exact and strong monoidal.\end{prop}

\begin{proof} For the categorical equivalence, see \cite[Proposition 1.4.4]{HTT}. Exactness is now clear and it suffices to show that $\zG(X,\bullet)$ is strong monoidal. We know that $\zG(X,\bullet)$ is strong monoidal as functor between modules over functions, see (\ref{ShVsSectqcMod1}). Hence, if $\cP,\cQ\in{\tt qcMod}(\cD_X)$, then \be\label{MonoidGlobSect}\zG(X,\cP\0_{\cO_X}\cQ)\simeq \zG(X,\cP)\0_{\cO_X(X)}\zG(X,\cQ)\ee as $\cO_X(X)$-modules. Recall now that we defined the $\cD_X$-module structure on $\cP\0_{\cO_X}\cQ$ by `extending' the $\Theta_X$-action (\ref{TensDMod}) on the presheaf $\cP\oslash_{\cO_X}\cQ$, see (\ref{TensSh}). In view of (\ref{MonoidGlobSect}), the action $\nabla^X$ of $\Theta_X(X)$ on $\cP(X)\0_{\cO_X(X)}\cQ(X)$ and {on} $(\cP\0_{\cO_X}\cQ)(X)$ `coincide', and so do the $\cD_X(X)$-module structures of these modules. {Finally}, the global section functor is strong monoidal. \end{proof}

\begin{rem} In the following, we work systematically over a smooth affine algebraic variety $X$ over an algebraically closed field $\K$ of characteristic 0.\end{rem}

Since the category ${\tt qcMod}(\cD_X)$ is abelian symmetric monoidal, the category ${\tt DG_+qcMod}(\cD_X)$ of differential non-negatively graded $\cO_X$-quasi-coherent $\cD_X$-modules is abelian and symmetric monoidal as well -- for the usual tensor product of chain complexes and chain maps. The unit of this tensor product is the chain complex $\cO_X$ concentrated in degree 0. The symmetry $\zb:\cP_\bullet\0 \cQ_\bullet\to \cQ_\bullet\0 \cP_\bullet$ is given by $$\zb(p\0 q)=(-1)^{\tilde p\tilde q}q\0 p\;,$$ where `tilde' denotes the degree and where the sign is necessary to obtain a chain map. Let us also mention that the zero object of ${\tt DG_+qcMod}(\cD_X)$ is the chain complex $(\{0\},0)\,$.

\begin{prop}\label{MonoidEquiv2} If $X$ is a smooth affine algebraic variety, its global section functor induces an equivalence \be\label{ShVsSectDGqcModDMon}\zG(X,\bullet):({\tt DG_+qcMod}(\cD_X),\0_{\cO_X},\cO_X)\to ({\tt DG_+Mod}(\cD_X(X)),\0_{\cO_X(X)},\cO_X(X))\;\ee of abelian symmetric monoidal categories, and is exact and strong monoidal.\end{prop}

\begin{proof} Let $F=\zG(X,\bullet)$ and $G$ be quasi-inverse (additive) functors that implement the equivalence (\ref{ShVsSectqcModDMon}). They induce functors $\mathbf{F}$ and $\mathbf{G}$ between the corresponding categories of chain complexes. Moreover, the natural isomorphism $a:\id\Rightarrow G\circ F$ induces, for each chain complex $\cP_\bullet\in{\tt DG_+qcMod}(\cD_X)$, a chain isomorphism $\mathbf{a}_{\cP_\bullet}:\cP_\bullet\rightarrow (\mathbf{G\circ F})(\cP_\bullet)$, which is functorial in $\cP_\bullet\,$. Both, the chain morphism property of $\mathbf{a}_{\cP_\bullet}$ and the naturality of $\mathbf{a}$, are direct consequences of the naturality of $a$ {-- since the action of $\mathbf{a}$ on a chain complex is given by the degreewise action of $a$}. Similarly, the natural isomorphism $b:F\circ G\Rightarrow \id$ induces a natural isomorphism $\mathbf{b}:\mathbf{F\circ G}\Rightarrow\id$, so that ${\tt DG_+qcMod}(\cD_X)$ and ${\tt DG_+Mod}(\cD_X(X))$ are actually equivalent categories. Since $F:{\tt qcMod}(\cD_X)\to{\tt Mod}(\cD_X(X))$ is strong monoidal and commutes with colimits (as left adjoint of $G$), it is straightforwardly checked that $\mathbf{F}$ is strong monoidal.\end{proof}

\subsection{Differential graded $\cD_X$-algebras vs. differential graded $\cD_X(X)$-algebras}

The strong monoidal functors $\mathbf{F}:{\tt DG_+qcMod}(\cD_X)\rightleftarrows{\tt DG_+Mod}(\cD_X(X)):\mathbf{G}$ yield an equivalence between the corresponding categories of commutative monoids:

\begin{cor} For any smooth affine variety $X$, there is an equivalence of categories \be\label{ShVsSectqcDAlg} \zG(X,\bullet): {\tt DG_+qcCAlg}(\cD_X)\rightarrow{\tt DG_+CAlg}(\cD_X(X))\;\ee between the category of differential graded quasi-coherent commutative $\cD_X$-algebras and the category of differential graded commutative $\cD_X(X)$-algebras.\end{cor}

The main goal of the present paper is to construct a model category structure on the {\small LHS} category. In view of the preceding corollary, it suffices to build this model structure on the {\small RHS} category. {We thus deal in the following exclusively with the category of {\bf differential graded $\cD$-algebras (resp., $\cD$-algebras)}, where $\cD:=\cD_X(X)$, which we denote simply by $\tt DG\cD A$ (resp., $\tt \cD A$). Similarly, the objects of ${\tt DG_+Mod}(\cD_X(X))$ (resp., ${\tt Mod}(\cD_X(X))$) are termed {\bf differential graded $\cD$-modules (resp., $\cD$-modules)} and their category is denoted by $\tt DG\cD M$ (resp., $\tt \cD M$).}

\subsection{The category $\tt DG\cD A$}

In this subsection we describe the category $\tt DG\cD A$ and prove first properties.\medskip

Whereas ${\op{Hom}_{\cD\tt M}(P,Q),}$ $P,Q\in{\tt \cD M}$, is a $\K$-vector space, the set ${\op{Hom}_{\cD{\tt A}}(A,B),}$ $A,B\in{\tt \cD A}$, is {not even} an abelian group. Hence, we cannot consider the category of chain complexes over commutative $\cD$-algebras and the objects of $\tt DG\cD A$ are (probably useless to say) no chain complexes of algebras.\medskip

As explained above, a $\cD$-algebra is a commutative unital $\cO$-algebra, endowed with a $\cD$-module structure {(which extends the $\cO$-module structure)}, such that vector fields act by derivations. Analogously, a differential graded $\cD$-algebra is easily seen to be a differential graded commutative unital $\cO$-algebra (a graded $\cO$-module together with an $\cO$-bilinear degree respecting multiplication, which is associative, unital, and graded-commutative; this module comes with a square 0, degree $-1$, $\cO$-linear, graded derivation), which is also a differential graded $\cD$-module (for the same differential, grading, and $\cO$-action), such that vector fields act as {degree zero} derivations.

\begin{prop} A differential graded $\cD$-algebra is a differential graded commutative unital $\cO$-algebra, as well as a differential graded $\cD$-module, such that vector fields act as derivations. Further, the morphisms of $\tt DG\cD A$ are the morphisms of $\tt DG\cD M$ that respect the multiplications and units. \end{prop}

In fact:

\begin{prop} The category $\tt DG\cD A$ is symmetric monoidal for the tensor product of $\tt DG\cD M$ with values on objects that are promoted canonically from $\tt DG\cD M$ to $\tt DG\cD A$ and same values on morphisms. The tensor unit is $\cO$; the initial object $(\,$resp., terminal object$\,)$ is $\cO$ $(\,$resp., $\{0\}$$\,)$.\end{prop}

\begin{proof} Let $A_\bullet,B_\bullet\in\tt DG\cD A$. Consider homogeneous vectors $a\in A_{\tilde{a}}$, $a'\in A_{\tilde{a}'}$, $b\in B_{\tilde{b}}$, $b'\in B_{\tilde{b}'}$, such that $\tilde a+\tilde b=m$ and $\tilde a'+\tilde b'=n$. Endow now the tensor product $A_\bullet\0_\cO B_\bullet\in\tt DG\cD M$ with the multiplication $\star$ defined by \be\label{MultTensMod}(A_\bullet\0_\cO B_\bullet)_m\times (A_\bullet\0_\cO B_\bullet)_n\ni(a\0 b,a'\0 b')\mapsto $$ $$(a\0 b)\star(a'\0 b') = (-1)^{\tilde a'\tilde b}(a\star_A a')\0 (b\star_B b')\in (A_\bullet\0_\cO B_\bullet)_{m+n}\;,\ee where the multiplications of $A_\bullet$ and $B_\bullet$ are denoted by $\star_A$ and $\star_B$, respectively. The multiplication $\star$ equips $A_\bullet\0_\cO B_\bullet$ with a structure of differential graded $\cD$-algebra. Note also that the multiplication of $A_\bullet\in\tt DG\cD A$ is a $\tt DG\cD A$-morphism $\zm_A:A_\bullet\0_\cO A_\bullet\to A_\bullet\,$.\medskip

Further, the unit of the tensor product in $\tt DG\cD A$ is the unit $(\cO,0)$ of the tensor product in $\tt DG\cD M$.\medskip

Finally, let $A_\bullet,B_\bullet,C_\bullet,D_\bullet\in\tt DG\cD A$ and let $\zf:A_\bullet\to C_\bullet$ and $\psi:B_\bullet\to D_\bullet$ be two $\tt DG\cD A$-morphisms. Then the $\tt DG\cD M$-morphism $\zf\0\psi:A_\bullet\0_\cO B_\bullet\to C_\bullet\0_\cO D_\bullet$ is also a $\tt DG\cD A$-morphism.\medskip

All these claims (as well as all the additional requirements for a symmetric monoidal structure) are straightforwardly checked.\medskip

The initial and terminal objects in $\tt DG\cD A$ are the differential graded $\cD$-algebras $(\cO,0)$ and $(\{0\},0)$, respectively. Indeed, in view of the adjunction (\ref{Adj}), the initial object of $\tt DG\cD A$ is the image by $\cS$ of the initial object of $\tt DG\cD M$.\end{proof}

Let us still mention the following

\begin{prop}\label{ProdTensProd} If $\zf:A_\bullet\to C_\bullet$ and $\psi:B_\bullet\to C_\bullet$ are $\tt DG\cD A$-morphisms, then $\chi:A_\bullet\0_\cO B_\bullet\to C_\bullet$, which is well-defined by $\chi(a\0 b)=\zf(a)\star_C \psi(b),$ is a $\tt DG\cD A$-morphism that restricts to $\zf$ (resp., $\psi$) on $A_\bullet$ (resp., $B_\bullet$).\end{prop}

\begin{proof} It suffices to observe that $\chi=\zm_C\circ(\zf\0 \psi)\,$.\end{proof}

\section{Finitely generated model structure on $\tt DG{\cD}M$}

When dealing with model categories, we use the definitions of \cite{Hov}. A short comparison of various definitions used in the literature can be found in Appendix \ref{ModCat} {below}. For additional information, we refer the reader to \cite{GS}, \cite{Hir}, \cite{Hov}, and \cite{Quill}.\medskip

Let us recall that $\tt DG\cD M$ is the category ${\tt Ch}_+(\cD)$ of non-negatively graded chain complexes of left modules over the non-commutative unital ring $\cD=\cD_X(X)$ of differential operators of a smooth affine algebraic variety $X$. The remaining part of this section actually holds for any not necessarily commutative unital ring $R$ and the corresponding category ${\tt Ch}_+(R)$. We will show that ${\tt Ch}_+(R)$ is a finitely (and thus cofibrantly) generated model category.\medskip

In fact, most of the familiar model categories are cofibrantly generated. For instance, in the model category $\tt SSet$ of simplicial sets, the generating cofibrations $I$ (resp., the generating trivial cofibrations $J$) are the canonical simplicial maps $\p\zD[n]\to \zD[n]$, whose sources are the boundaries of the standard simplicial $n$-simplices (resp., the canonical maps $\zL^r[n]\to \zD[n]$, whose sources are the $r$-horns of the standard $n$-simplices, $0\le r\le n$). The generating cofibrations and trivial cofibrations of the model category $\tt Top$ of topological spaces -- which is Quillen equivalent to $\tt SSet$ -- are defined similarly. The homological situation is analogous to the topological and combinatorial ones. In the case of ${\tt Ch}_+(R)$, the set $I$ of generating cofibrations (resp., the set $J$ of generating trivial cofibrations) is made (roughly) of the maps $S^{n-1}\to D^n$ from the $(n-1)$-sphere to the $n$-disc (resp., of the maps $0\to D^n$). In fact, the $n$-disc $D^n$ is the chain complex \be\label{Disc}D^n_{\bullet}: \cdots \to 0\to 0\to \stackrel{(n)}{R} \to \stackrel{(n-1)}{R}\to 0\to \cdots\to \stackrel{(0)}{0}\;,\ee whereas the $n$-sphere $S^n$ is the chain complex \be\label{Sphere}S^n_\bullet: \cdots \to 0\to 0\to \stackrel{(n)}{R}\to 0\to \cdots\to \stackrel{(0)}{0}\;.\ee Definition (\ref{Disc}), in which the differential is necessarily the identity of $R$, is valid for $n\ge 1$. Definition (\ref{Sphere}) makes sense for $n\ge 0$. We extend the first (resp., second) definition to $n=0$ (resp., $n=-1$) by setting $D^0_\bullet:=S^0_\bullet$ (resp., $S^{-1}_\bullet:=0_\bullet$). The chain maps $S^{n-1}\to D^n$ are canonical (in degree $n-1$, they necessarily coincide with $\id_R$), and so are the maps $0\to D^n$. We now define the set $I$ (resp., $J$) by \be\label{GenCof} I=\{\iota_n: S^{n-1} \to D^n, n\ge 0\}\ee $(\,$resp., \be\label{GenTrivCof}J=\{\zeta_n: 0 \to D^n, n\ge 1\}\;)\;.\ee

\begin{theo}\label{FinGenModDGDM}
For any unital ring $R$, the category ${\tt Ch}_+(R)$ of non-negatively graded chain complexes of left $R$-modules is a finitely $(\,$and thus a cofibrantly$\,)$ generated model category $(\,$in the sense of \cite{GS} and in the sense of \cite{Hov}$\,)$, with $I$ as its generating set of cofibrations and $J$ as its generating set of trivial cofibrations. The weak equivalences are the maps that induce isomorphisms in homology, the cofibrations are the injective maps with degree-wise projective cokernel $(\,$projective object in ${\tt Mod}(R)$$\,)$, and the fibrations are the maps that are surjective in $(\,$strictly$\,)$ positive degrees. Further, the trivial cofibrations are the injective maps $i$ whose cokernel $\coker(i)$ is strongly projective as a chain complex $(\,$strongly projective object $\coker(i)$ in ${\tt Ch}_+(R)$, in the sense that, for any map $c:\coker(i)\to C$ and any map $p:D\to C$, there is a map $\ell:\coker(i)\to D$ such that $p\circ\ell=i$, if $p$ is surjective in $(\,$strictly$\,)$ positive degrees$\,)$.\end{theo}

\begin{proof}

{The following proof uses the definitions of (cofibrantly generated) model categories used in \cite{DS} and \cite{GS}, as well as the non-equivalent definitions of these concepts given in \cite{Hov}: we refer again to the Appendix \ref{ModCat} below}.\medskip

It is known that ${\tt Ch}_+(R)$, with the described weak equivalences, cofibrations, and fibrations is a model category (Theorem 7.2 in \cite{DS}). A model category in the sense of \cite{DS} contains all finite limits and colimits; the $\op{Cof}-\op{TrivFib}$ and $\op{TrivCof}-\op{Fib}$ factorizations are neither assumed to be functorial, nor, of course, to be chosen functorial factorizations. Moreover, we have $\op{Fib}=\op{RLP}(J)$ and $\op{TrivFib}=\op{RLP}(I)$ (Proposition 7.19 in \cite{DS}).\medskip

Note first that ${\tt Ch}_+(R)$ has all small limits and colimits, which are taken degree-wise.\medskip

Observe also that the domains and codomains $S^n$ ($n\ge 0$) and $D^n$ ($n\ge 1$) of the maps in $I$ and $J$ are bounded chain complexes of finitely presented $R$-modules (the involved modules are all equal to $R$). However, every bounded chain complex of finitely presented $R$-modules is $n$-small, $n\in\N$, relative to all chain maps (Lemma 2.3.2 in \cite{Hov}). Hence, the domains and codomains of $I$ and $J$ satisfy the smallness condition of a finitely generated model category, and are therefore small in the sense of the finite and transfinite definitions of a cofibrantly generated model category.\medskip

It thus follows from the Small Object Argument that there exist in ${\tt Ch}_+(R)$ a functorial $\op{Cof}-\op{TrivFib}$ and a functorial $\op{TrivCof}-\op{Fib}$ factorization. Hence, the first part of Theorem \ref{FinGenModDGDM}.\medskip

As for the part on trivial cofibrations, its proof is the same as the proof of Lemma 2.2.11 in \cite{Hov}.\end{proof}

In view of Theorem \ref{FinGenModDGDM}, let us recall that any projective chain complex $(K,d)$ is degree-wise projective. Indeed, consider, for $n\ge 0$, an $R$-linear map $k_n:K_n\to N$ and a surjective $R$-linear map $p:M\to N$, and denote by $D^{n+1}(N)$ (resp., $D^{n+1}(M)$) the disc defined as in (\ref{Disc}), except that $R$ is replaced by $N$ (resp., $M$). Then there is a chain map $k:K\to D^{n+1}(N)$ (resp., a surjective chain map $\zp:D^{n+1}(M)\to D^{n+1}(N)$) that is zero in each degree, except in degree $n+1$ where it is $k_n\circ d_{n+1}$ (resp., $p$) and in degree $n$ where it is $k_n$ (resp., $p$). Since $(K,d)$ is projective as chain complex, there is a chain map $\ell:K\to D^{n+1}(M)$ such that $\zp\circ \ell =k$. In particular, $\ell_n:K_n\to M$ is $R$-linear and $p\circ \ell_n=k_n\,.$

\section{Finitely generated model structure on $\tt DG\cD A$}

\subsection{Adjoint functors between $\tt DG\cD M$ and $\tt DG\cD A$}\label{Adjunction}

We aim at transferring to $\tt DG\cD A$ the just described finitely generated model structure on $\tt DG\cD M$. Therefore, we need a pair of adjoint functors.

\begin{prop} The graded symmetric tensor algebra functor $\cS$ and the forgetful functor $\op{For}$ provide an adjoint pair \be\label{Adj}{\cal S}:{\tt DG\cD M}\rightleftarrows{\tt DG\cD A}:\op{For}\;\ee between the category of differential graded $\cD$-modules and the category of differential graded $\cD$-algebras.\end{prop}

\begin{proof} For any $M_\bullet\in{\tt DG\cD M}$, the sum $$\0_\cO^\ast M_\bullet=\cO\oplus\bigoplus_{n\ge 1}M_\bullet^{\0_\cO n}\in{\tt DG\cD M}\;$$ is the free associative unital $\cO$-algebra over the $\cO$-module $M_\bullet\,.$ When passing to graded symmetric tensors, we divide by the obvious $\cO$-ideal $\cI$, which is further a sub {\small DG} $\cD$-module. Therefore, the free graded symmetric unital $\cO$-algebra \be\label{Alg1}{\cal S}_\cO^\ast M_\bullet=\0_\cO^\ast M_\bullet/{\cal I}\;,\ee with multiplication $[S]\odot [T]=[S\0 T]\,$, is also a {\small DG} $\cD$-module. It is straightforwardly checked that ${\cal S}_\cO^\ast M_\bullet\in {\tt DG\cD A}$. The definition of $\cS$ on morphisms is obvious.\medskip

As concerns the proof that the functors $\op{For}$ and $\cal S$ are adjoint, i.e., that \be\label{Adjoint}\h_{\tt DG\cD A}({\cal S}_\cO^\ast M_\bullet,A_\bullet)\simeq \h_{\tt \tt DG\cD M}(M_\bullet,\op{For}A_\bullet)\;,\ee functorially in $M_\bullet\in{\tt DG\cD M}$ and $A_\bullet\in{\tt DG\cD A}\,$, let $\zf:M_\bullet\to \op{For}A_\bullet$ be a $\tt DG\cD M$-map. Since ${\cal S}_\cO^\ast M_\bullet$ is free in the category $\tt GCA$ of graded commutative associative unital graded $\cO$-algebras, a $\tt GCA$-morphism is completely determined by its restriction to the graded $\cO$-module $M_\bullet\,$. Hence, the extension $\bar\zf:{\cal S}_\cO^\ast M_\bullet\to A_\bullet$ of $\zf$, defined by $\bar\zf(1_\cO)=1_{A}$ and by $$\bar\zf(m_1\odot\ldots\odot m_k)=\zf(m_1)\star_A\ldots\star_A\zf(m_k)\;,$$ is a $\tt GCA$-morphism. This extension is also a $\tt DG\cD A$-map, i.e., a $\tt DG\cD M$-map that respects the multiplications and the units, if it intertwines the differentials and is $\cD$-linear. These requirements, as well as functoriality, are straightforwardly checked. \end{proof}

Recall that a free object in a category $\tt D$ over an object $C$ in a category $\tt C$, such that there is a forgetful functor $\op{For}:\tt D\to C$, is a universal pair $(F(C),i)$, where $F(C)\in\tt D$ and $i\in\op{Hom}_{\tt C}(C,\op{For} F(C))\,$.

\begin{rem}\label{FreeDGDA} Equation (\ref{Adjoint}) means that ${\cal S}_\cO^\star M_\bullet$ is the {\bf free differential graded $\cD$-algebra} over the differential graded $\cD$-module $M_\bullet\,$.\end{rem}

A definition of $\cS_\cO^\ast M_\bullet$ via invariants can be found in Appendix \ref{InvCoinv}.

\subsection{Relative Sullivan $\cD$-algebras}\label{RSDA}

If $V_\bullet$ is a non-negatively graded $\cD$-module and $(A_\bullet,d_A)$ a differential graded $\cD$-algebra, the tensor product $A_\bullet\0_\cO\cS^\star_\cO V_\bullet$ is a graded $\cD$-algebra. In the following definition, we assume that this algebra is equipped with a differential $d$, such that $$(A_\bullet\0_\cO\cS^\star_\cO V_\bullet,d)\in\tt DG\cD A$$ contains $(A_\bullet,d_A)$ as sub-{\small DG$\cD$A}. The point is that $(A_\bullet,d_A)$ is a differential submodule of the tensor product differential module, but that usually the module $\cS^\star_\cO V_\bullet$ is not. The condition that $(A_\bullet,d_A)$ be a sub-{\small DG$\cD$A} can be rephrased by asking that the inclusion $$A_\bullet\ni a\mapsto a\0 1\in A_\bullet\0_\cO\cS^\star_\cO V_\bullet\;$$ be a $\tt DG\cD A$-morphism. This algebra morphism condition or subalgebra condition would be automatically satisfied {if the differential $d$ on $A_\bullet\0_\cO\cS^\star_\cO V_\bullet$ were} defined by \be\label{split}d=d_A\0 \id + \id\0 d_\cS\;,\ee where $d_\cS$ is a differential on $\cS^\star_\cO V_\bullet$ (in particular the differential $d_\cS=0$). However, as mentioned, this is generally not the case.\medskip

We omit in the following $\bullet,$ $\star,$ as well as subscript $\cO$, provided clarity does not suffer hereof. Further, to avoid confusion, we sometimes substitute {$\boxtimes$ for $\0$} to emphasize that the differential $d$ of $A\boxtimes\cS V$ is not necessarily obtained from the differential $d_A$ and a differential $d_\cS$.\footnote{{Such twisted differentials typically appear when one adds new generators to improve homological properties and in particular to kill homology in lower degrees.}}\medskip

{We now give the $\cD$-algebraic version of the definition of a relative Sullivan algebra \cite{FHT}. Note that the factorizations that are considered in \cite{FHT} are not, as the factorizations here below, obtained via pushouts and are not functorial.}

\begin{defi}\label{RSullDAlg} A {\bf relative Sullivan $\cD$-algebra} $(\,${\small RS$\cD\!$A}$\,)$ is a $\tt DG\cD A$-morphism
$$(A,d_A)\to(A\boxtimes \cS V,d)\;$$ that sends $a\in A$ to $a\0 1\in A\boxtimes \cS V$. Here $V$ is a free non-negatively graded $\mathcal{D}$-module $$V=\bigoplus_{\za\in J}\,\cD\cdot v_\za\;,$$ which admits a homogeneous basis $(v_\za)_{\za\in J}$ that is indexed by a well-ordered set $J$, and is such that \be\label{Lowering}d v_\za \in A\boxtimes \cS V_{<\za}\;,\ee for all $\za\in J$. In the last requirement, we set $V_{<\za}:=\bigoplus_{\zb<\za}\cD\cdot v_\zb\,$. We refer to Property (\ref{Lowering}) by saying that $d$ is {\bf lowering}.

A {\small RS$\cD\!$A} with Property (\ref{split}) $(\,$resp., over $(A,d_A)=(\cO,0)$$\,)$ is called a {\bf split} {\small RS$\cD\!$A} $(\,$resp., a {\bf Sullivan $\cD$-algebra} $(\,${\small S$\cD\!$A}$\,)$$\,)$  and it is often simply denoted by $(A\otimes \cS V,d)$ $(\,$resp.,\;$(\cS V,d)$$\,)$.
\end{defi}

The next two lemmas are of interest for the split situation.

\begin{lem}\label{DiffGen} Let $(v_{\za})_{\za\in I}$ be a family of generators of homogeneous non-negative degrees, and let $$V:=\langle v_\za: \za\in I\ra:=\bigoplus_{\za\in I}\,\cD\cdot v_\za$$ be the free non-negatively graded $\cD$-module over $(v_\za)_{\za\in I}$. Then, any degree $-1$ map $d\in {\tt Set}((v_\za),V)$ uniquely extends to a degree $-1$ map $d\in{\tt \cD M}(V,V)$. If moreover $d^2=0$ on $(v_\za)$, then $(V,d)\in\tt DG\cD M\,.$  \end{lem}

Since $\cS V$ is the free differential graded $\cD$-algebra over the differential graded $\cD$-module $V$, a morphism $f\in {\tt DG\cD A}(\cS V,B),$ valued in $(B,d_B)\in {\tt DG\cD A}$, is completely defined by its restriction $f\in {\tt DG\cD M}(V,B)$. Hence, the

\begin{lem}\label{MorpGen} Consider the situation of Lemma \ref{DiffGen}. Any degree 0 map $f\in {\tt Set}((v_\za), B)$ uniquely extends to a morphism $f\in{\tt G\cD M}(V,B)$. Furthermore, if $d_B\,f=f\,d$ on $(v_\za)$, this extension is a morphism $f\in{\tt DG\cD M}(V,B),$ which in turn admits a unique extension $f\in{\tt DG\cD A}(\cS V,B)$.\end{lem}

\subsection{Quillen's transfer theorem}

We use the adjoint pair \be\label{Ad}{\cal S}:{\tt DG\cD M}\rightleftarrows{\tt DG\cD A}:\op{For}\;\ee to transfer the cofibrantly generated model structure from the source category $\tt DG\cD M$ to the target category $\tt DG\cD A$. This is possible if Quillen's transfer theorem \cite{Quill} applies.

\begin{theo}\label{QTT}
Let $ F : {\tt C} \rightleftarrows {\tt D} : G $ be a pair of adjoint functors. Assume that $\tt C$ is a cofibrantly generated model category and denote by $I$ (resp., $J$) its set of generating cofibrations (resp., trivial cofibrations). Define a morphism $f : X \to Y$ in $\tt D$ to be a weak equivalence (resp., a fibration), if $Gf$ is a weak equivalence (resp., a fibration) in $\tt C$. If
\begin{enumerate}
\item
the right adjoint $G : {\tt D} \to {\tt C}$ commutes with sequential colimits, and
\item
any map in $\tt D$ with the {\small LLP} with respect to all fibrations is a weak equivalence,
\end{enumerate}
then $\tt D$ is a cofibrantly generated model category that admits $\{ Fi: i \in I \}$ (resp., $\{ Fj : j \in J \}$) as set of generating cofibrations (resp., trivial cofibrations).
\end{theo}

Of course, in this version of the transfer principle, the mentioned model structures are cofibrantly generated model structures in the sense of \cite{GS}.\medskip

Condition 2 is the main requirement of the transfer theorem. It can be checked using the following lemma \cite{Quill}:

\begin{lem}[Quillen's path object argument]\label{SuffCondFor2}
Assume in a category $\tt D$ (which is not yet a model category, but has weak equivalences and fibrations),
\begin{enumerate}
\item there is a functorial fibrant replacement functor, and
\item every object has a natural path object, i.e., for any $D\in \tt D$, we have a natural commutative
diagram

\begin{center}
\begin{tikzpicture}
\node(C){$D$};
\node(A)[right of=C, xshift=2cm]{$D\times D$};
\node(P)[above of=A,yshift=2cm]{$\op{Path}(D)$};
\draw[->](C) to node[above]{${{\Delta}}$} (A);
\draw[->](C) to node{$\hspace{-30pt}i$} (P);
\draw[->](P) to node[below]{$\hspace{17pt} q$} (A);
\end{tikzpicture}
\end{center}

\end{enumerate}
where $\zD$ is the diagonal map, $i$ is a weak equivalence and $q$ is a fibration.
Then every map in $\tt D$ with the {\small LLP} with respect to all fibrations is a weak
equivalence.
\end{lem}

We think about $\op{Path}(D)\in\tt D$ is an internalized `space' of paths in $D$. In simple cases, $\op{Path}(D)=\h_{\tt D}(I,D)$, where $I\in\tt D$ and where $\h_{\tt D}$ is an internal Hom. Moreover, by fibrant replacement of an object $D\in\tt D$, we mean a weak equivalence $D\to \bar{D}$ whose target is a fibrant object.

\subsection{Proof of Condition 1 of Theorem \ref{QTT}}

Let $\zl$ be a non-zero ordinal and let $X:\zl\to \tt C$ be a diagram of type $\zl$ in a category $\tt C$, i.e., a functor from $\zl$ to $\tt C$. Since an ordinal number is a totally ordered set, the considered ordinal $\zl$ can be viewed as a directed poset $(\zl,\le)$. Moreover, the diagram $X$ is a direct system in $\tt C$ over $\zl$ -- made of the $\tt C$-objects $X_\zb$, $\zb<\zl$, and the $\tt C$-morphisms $X_{\zb\zg}:X_\zb\to X_\zg$, $\zb\le\zg$, and the colimit $\colim_{\zb<\zl}X_\zb$ of this diagram $X$ is the inductive {limit to the} system $(X_\zb,X_{\zb\zg})$.\medskip

Let now $A:\zl\to \tt DG\cD A$ be a diagram of type $\zl$ in $\tt DG\cD A$ and let $\op{For}\circ A:\zl\to \tt DG\cD M$ be the corresponding diagram in $\tt DG\cD M$. To simplify notation, we denote the latter diagram simply by $A$. As mentioned in the proof of Theorem \ref{FinGenModDGDM}, the colimit of $A$ does exist in $\tt DG\cD M$ and is taken degree-wise in ${\tt Mod}(\cD)$. For any degree $r\in\N$, the colimit $C_r$ of the functor $A_r:\zl\to\tt Mod(\cD)$ is the inductive {limit in $\tt Mod(\cD)$ to the} direct system $(A_{\zb,r}, A_{\zb\zg,r})$ -- which is obtained via the usual construction in $\tt Set$. Due to universality, one naturally gets a ${\tt Mod}(\cD)$-morphism $d_{r}:C_r\to C_{r-1}$. The complex $(C_\bullet,d)$ is the colimit in $\tt DG\cD M$ of $A$. It is now straightforwardly checked that the canonical multiplication $\diamond$ in $C_\bullet$ provides an object $(C_\bullet,d,\diamond)\in \tt DG\cD A$ and that this object is the colimit of $A$ in $\tt DG\cD A$.\medskip

Hence, the

\begin{prop}\label{Enrichment} Let $\zl$ be a non-zero ordinal. The forgetful functor $\op{For}:\tt DG\cD A\to DG\cD M$ creates colimits of diagrams of type $\lambda$ in $\tt DG\cD A$, i.e., for any diagram $A$ of type $\lambda$ in $\tt DG\cD A$, we have \be\label{EnrichmentEq}\op{For}(\op{colim}_{\zb<\zl}A_{\zb,\bullet})=\op{colim}_{\zb<\zl}\op{For}(A_{\zb,\bullet})\;.\ee\end{prop}

If $\zl$ is the zero ordinal, it can be viewed as the empty category $\emptyset$. Therefore, the colimit in $\tt DG\cD A$ of the diagram of type $\zl$ is in this case the initial object $(\cO,0)$ of $\tt DG\cD A$. Since the initial object in $\tt DG\cD M$ is $(\{0\},0)$, we see that $\op{For}$ does not commute with this colimit. The above proof fails indeed, as $\emptyset$ is not a directed set.\medskip

It follows from Proposition \ref{Enrichment} that the right adjoint $\op{For}$ in (\ref{Ad}) commutes with sequential colimits, so that the first condition of Theorem \ref{QTT} is satisfied.

\begin{rem} Since a right adjoint functor between accessible categories preserves all filtered colimits, the first condition of Theorem \ref{QTT} is a consequence of the accessibility of $\tt DG\cD M$ and $\tt DG\cD A$. We gave a direct proof to avoid the proof of the accessibility of $\tt DG\cD A$.\end{rem}

\subsection{Proof of Condition 2 of Theorem \ref{QTT}}\label{Condition2}

We prove Condition 2 using Lemma \ref{SuffCondFor2}. In our case, the adjoint pair is $${\cal S}:{\tt DG\cD M}\rightleftarrows{\tt DG\cD A}:\op{For}\;.$$ As announced in Subsection \ref{RSDA}, we omit $\bullet$, $\star$, and $\cO$, whenever possible. It is clear that every object $A\in{\tt D}={\tt DG\cD A}$ is fibrant. Hence, we can choose the identity as fibrant replacement functor, with the result that the latter is functorial.\medskip

As for the second condition of the lemma, we will show that {\it any} $\,\tt DG\cD A$-morphism $\zf:A\to B$ naturally factors into a weak equivalence followed by a fibration.\medskip

Since in the standard model structure on the category of differential graded commutative algebras over $\Q$, cofibrations are retracts of relative Sullivan algebras \cite{Hes}, the obvious idea is to decompose $\zf$ as $A\to A\0\cS V\to B$, where $i: A\to A\0\cS V$ is a (split) relative Sullivan $\cD$-algebra, such that there is a projection $p: A\0\cS V\to B$, or, even better, a projection $\ze: V\to B$ in positive degrees. The first attempt might then be to use $$\ze:V=\bigoplus_{n>0}\bigoplus_{b_n\in B_n}\cD\cdot 1_{b_n}\ni 1_{b_n}\mapsto b_n\in B\;,$$ whose source incorporates a copy of the sphere $S^n$ for each $b_n\in B_n$, $n>0\,.$ However, $\ze$ is not a chain map, since in this case we would have $d_Bb_n=d_B \ze 1_{b_n}=0$, for all $b_n$. The next candidate is obtained by replacing $S^n$ by $D^n$: if $B\in {\tt DG\cD M}$, set $$P(B)=\bigoplus_{n>0}\bigoplus_{b_n\in B_n}D^n_{\bullet}\in {\tt DG\cD M}\;,$$ where $D^n_{\bullet}$ is a copy of the $n$-disc $$D^n_{\bullet}: \cdots \to 0\to 0\to \cD\cdot \mathbb{I}_{b_n} \to \cD\cdot s^{-1}\mathbb{I}_{b_n}\to 0\to \cdots\to 0\;.$$ Since $$P_n(B)=\bigoplus_{b_{n+1}\in B_{n+1}}\cD\cdot s^{-1}\mathbb{I}_{b_{n+1}}\oplus \bigoplus_{b_n\in B_n}\cD\cdot \mathbb{I}_{b_n}\;\; (n>0)\quad\text{and}\quad P_0(B)=\bigoplus_{b_1\in B_1}\cD\cdot s^{-1}\mathbb{I}_{b_1}\;,$$ the free non-negatively graded $\cD$-module $P(B)$ is projective in each degree, what justifies the chosen notation. On the other hand, the differential $d_P$ of $P(B)$ is the degree $-1$ square 0 $\cD$-linear map induced by the differentials in the $n$-discs and thus defined on $P_n(B)$ by $$d_P(s^{-1}\mathbb{I}_{b_{n+1}})=0\in P_{n-1}(B)\quad\text{and}\quad d_P(\mathbb{I}_{b_n})=s^{-1}\mathbb{I}_{b_n}\in P_{n-1}(B)\;$$ (see Lemma \ref{DiffGen}). The canonical projection $\ze:P(B)\to B\,$, is defined on $P_n(B)$, as degree 0 $\cD$-linear map, by $$\ze(s^{-1}\mathbb{I}_{b_{n+1}})=d_B(b_{n+1})\in B_n\quad\text{and}\quad\ze(\mathbb{I}_{b_n})=b_n\in B_n\;.$$ It is clearly a $\tt DG\cD M$-morphism and extends to a $\tt DG\cD A$-morphism $\ze:\cS(P(B))\to B$ (see Lemma \ref{MorpGen}).\medskip

We define now the aforementioned $\tt DG\cD A$-morphisms $i:A\to A\0 \cS(P(B))$ and $p:A\0 \cS(P(B))\to B$, where $i$ is a weak equivalence and $p$ a fibration such that $p\circ i=\zf\,.$ We set $i=\id_A\0 1$ and $p=\zm_B\circ (\zf\0 \ze)\,.$ It is readily checked that $i$ and $p$ are $\tt DG\cD A$-morphisms (see Proposition \ref{ProdTensProd}) with composite $p\circ i=\zf\,.$ Moreover, by definition, $p$ is a fibration in $\tt DG\cD A$, if it is surjective in degrees $n>0$ -- what immediately follows from the fact that $\ze$ is surjective in these degrees.\medskip

It thus suffices to show that $i$ is a weak equivalence in $\tt DG\cD A$, i.e., that $$H(i):H(A)\ni [a]\to [a\0 1]\in H\left(A\0{\cal S}(P(B))\right)$$ is an isomorphism of graded $\cD$-modules. Since $\tilde \imath:A\to A\0\cO$ is an isomorphism in $\tt DG\cD M$, it induces an isomorphism $$H(\tilde \imath):H(A)\ni[a]\to [a\0 1]\in H(A\0\cO)\;.$$ In view of the graded $\cD$-module isomorphism $$H(A\0{\cal S}(P(B)))\simeq H( A\0\cO)\oplus H(A\0{\cal S}^{\ast\ge 1}(P(B)))\;,$$ we just have to prove that \be H(A\0{\cal S}^{k\ge 1}(P(B)))=0\;\label{SCond1}\ee as graded $\cD$-module, or, equivalently, as graded $\cO$-module.\medskip

To that end, note that $$0\longrightarrow \ker^{k}{\frak S}\stackrel{\iota}{\longrightarrow}P(B)^{\0 k}\stackrel{\frak S}{\longrightarrow}(P(B)^{\0 k})^{\mathbb{S}_k}\longrightarrow 0\;,$$ where $k\ge 1$ and where $\frak S$ is the averaging map, is a short exact sequence in the abelian category $\tt DG\cO M$ of differential non-negatively graded $\cO$-modules (see Appendix \ref{InvCoinv}, in particular Equation (\ref{SymOp})). Since it is canonically split by the injection $${\frak I}:(P(B)^{\0 k})^{\mathbb{S}_k}\to P(B)^{\0 k}\;,$$ and $$(P(B)^{\0 k})^{\mathbb{S}_k}\simeq {\cal S}^{k}(P(B))$$ as {\small DG} $\cO$-modules (see Equation (\ref{Alg2})), we get $$P(B)^{\0 k}\simeq \cS^k(P(B))\oplus \ker^k{\frak S}\quad\text{and}\quad A\0 P(B)^{\0 k}\simeq A\0\cS^k(P(B))\,\oplus\, A\0\ker ^k{\frak S}\;,$$ as {\small DG} $\cO$-modules. Therefore, it suffices to show that the {\small LHS} is an acyclic chain complex of $\cO$-modules.\medskip

We begin showing that $\cD=\cD_X(X)$, where $X$ is a smooth affine algebraic variety, is a flat module over $\cO=\cO_X(X)$. Note first that, the equivalence (\ref{ShVsSectqcMod1}) $$\zG(X,\bullet):{\tt qcMod}(\cO_X)\rightleftarrows {\tt Mod}(\cO):\widetilde{\bullet}$$ is exact and strong monoidal (see remark below Equation (\ref{ShVsSectqcMod1})). Second, observe that $\cD_X$ is a locally free $\cO_X$-module, hence, a flat (and quasi-coherent) sheaf of $\cO_X$-modules, i.e., $\cD_X\0_{\cO_X}\bullet\,$ is exact in ${\tt Mod}(\cO_X)$. To show that $\cD\0_\cO\bullet$ is exact in ${\tt Mod}(\cO)$, consider an exact sequence $$0\to M'\to M\to M''\to 0$$ in ${\tt Mod}(\cO)$. From what has been said it follows that $$0\to \cD_X\0_{\cO_X}\widetilde{M'}\to \cD_X\0_{\cO_X}\widetilde{M}\to \cD_X\0_{\cO_X}\widetilde{M''}\to 0$$ is an exact sequence in ${\tt Mod}(\cO_X)$, as well as an exact sequence in ${\tt qcMod}(\cO_X)$ (kernels and cokernels of morphisms of quasi-coherent modules are known to be quasi-coherent). When applying the exact and strong monoidal global section functor, we see that $$0\to\cD\0_\cO M'\to \cD\0_\cO M\to \cD\0_\cO M''\to 0$$ is exact in ${\tt Mod}(\cO)$.\medskip

Next, observe that
$$
H(A\0 P(B)^{\0 k})=\bigoplus_{n>0}\bigoplus _{b_n\in B_n} H(D^n_\bullet\0 A \otimes P(B)^{\otimes (k-1)})\;.
$$
To prove that each of the summands of the {\small RHS} vanishes, we apply K\"unneth's Theorem \cite[Theorem 3.6.3]{Wei93} to the complexes $D_\bullet^n$ and $A \otimes P(B)^{\otimes (k-1)}$, noticing that both, {the $n$-disc $D^n_\bullet$ (which vanishes, except in degrees $n,n-1$, where it coincides with $\cD$) and its boundary $d(D^n_\bullet)$} (which vanishes, except in degree $n-1$, where it coincides with $\cD$), are termwise flat $\mathcal{O}$-modules. We thus get, for any $m$, a short exact sequence
\begin{align}\nonumber
0\rightarrow \bigoplus_{p+q=m}H_p(D^n_\bullet)\otimes H_q(A \otimes P(B)^{\otimes (k-1)})&\rightarrow H_m(D_\bullet^n\0 A \otimes P(B)^{\otimes (k-1)})\rightarrow\\ \nonumber
&\bigoplus_{p+q=m-1} \op{Tor}_1(H_p(D_\bullet^n), H_q(A \otimes P(B)^{\otimes (k-1)}))\rightarrow 0\;.
\end{align}
Finally, since $D^n_\bullet$ is acyclic, the central term of this exact sequence vanishes, since both, the first and the third, do.\medskip

To completely finish checking the requirements of Lemma \ref{SuffCondFor2} and thus of Theorem \ref{QTT}, we still have to prove that the factorization $(i,p)=(i(\zf),p(\zf))$ of $\zf$ is functorial. In other words, we must show that, for any commutative $\tt DG\cD A$-square \be\label{InitSq}
\xymatrix{
A\ar[d]^{u} \ar[r]^{\zf}&B\ar[d]^{v\;\;\;,}\\
A'\ar[r]^{\zf'}&B'\\
}
\ee
there is a commutative $\tt DG\cD A$-diagram
\be\label{CompMor00}
\xymatrix{A\;\; \ar[d]_{u} \ar^{\sim}_{i(\zf)}  @{->} [r] & A\0\cS U \ar[d]^{w}\;\;\ar @{->>} [r]_{p(\zf)} & B \ar[d]^{v\;\;\;,}\\
A'\;\; \ar @{->} [r]^{\sim}_{i(\zf')} & A'\0\cS U'\;\; \ar @{->>} [r]_{p(\zf')} & B'\\
}
\ee
where we wrote $U$ (resp., $U'$) instead of $P(B)$ (resp., $P(B')$).

To construct the $\tt DG\cD A$-morphism $w$, we first define a $\tt DG\cD A$-morphism $\tilde{v}:\cS U\to \cS U'$, then we obtain the $\tt DG\cD A$-morphism $w$ by setting $w=u\0 \tilde{v}$.

To get the $\tt DG\cD A$-morphism $\tilde{v}$, it suffices, in view of Lemma \ref{MorpGen}, to define a degree 0 $\tt Set$-map $\tilde{v}$ on $G:=\{s^{-1}\mbi_{b_n},\mbi_{b_n}:b_n\in B_n,n>0\}$, with values in the differential graded $\cD$-algebra $(\cS U',d_{U'})$, which satisfies $d_{U'}\,\tilde{v}=\tilde{v}\,d_U$ on $G$. We set $$\tilde{v}(s^{-1}\mbi_{b_n})=s^{-1}\mbi_{v(b_n)}\in\cS U'\;\,\text{and}\;\,\tilde{v}(\mbi_{b_n})=\mbi_{v(b_n)}\in\cS U'\;,$$ and easily see that all the required properties hold.

We still have to verify that the diagram (\ref{CompMor00}) actually commutes. Commutativity of the left square is obvious. As for the right square, let $t:={a}\0 x_1\odot\ldots\odot x_k\in A\0\cS U$, where the $x_i$ are elements of $U$, and note that $$v\, p(\zf)(t)= v\, (\zm_B\circ (\zf\0 \ze))(t)=v\,\zf({a})\star v\,\ze(x_1)\star\ldots\star v\,\ze(x_k)$$ and $$p(\zf')w(t)=(\zm_{B'}\circ(\zf'\0\ze'))(u({a})\0 \tilde{v}(x_1)\odot\ldots\odot \tilde{v}(x_k))$$ $$=\zf'u({a})\star\,\ze'\, \tilde{v}(x_1)\,\star\,\ldots\,\star\, \ze'\, \tilde{v}(x_k)\;,$$ where $\star$ denotes the multiplication in $B'$. Since the square (\ref{InitSq}) commutes, it suffices to check that \be\label{ComRel}v\,\ze(x)=\ze'\,\tilde{v}(x)\;,\ee for any $x\in U\,.$ However, the $\cD$-module $U$ is freely generated by $G$ and the four involved morphisms are $\cD$-linear: it is enough that (\ref{ComRel}) holds on $G$ -- what is actually the case.

\subsection{Transferred model structure}

We proved in Theorem \ref{FinGenModDGDM} that $\tt DG\cD M$ is a finitely generated model category whose set of generating cofibrations (resp., trivial cofibrations) is \be\label{GenCof1} I=\{\iota_k: S^{k-1}_\bullet \to D^k_\bullet, k\ge 0\}\ee $(\,$resp., \be\label{GenTrivCof1}J=\{\zeta_k: 0 \to D^k_\bullet, k\ge 1\}\;)\;.\ee Theorem \ref{QTT} thus allows us to conclude that:

\begin{theo}\label{FinGenModDGDA} The category $\tt DG\mathcal{D}A$ of differential non-negatively graded commutative $\cD$-algebras is a finitely $(\,$and thus a cofibrantly$\,)$ generated model category $(\,$in the sense of \cite{GS} and in the sense of \cite{Hov}$\,)$, with $\cS I=\{\cS \iota_k:\iota_k\in I\}$ as its generating set of cofibrations and $\cS J=\{\cS \zeta_k: \zeta_k\in J\}$ as its generating set of trivial cofibrations. The weak equivalences are the $\tt DG\cD A$-morphisms that induce an isomorphism in homology. The fibrations are the $\tt DG\cD A$-morphisms that are surjective in all positive degrees $p>0$.\end{theo}

The cofibrations will be described below.\medskip

{Quillen's transfer principle actually provides a \cite{GS}-cofibrantly-generated (hence, a \cite{Hov}-cofibrantly-generated) \cite{GS}-model structure on $\tt DG\cD A$ (hence, a \cite{Hov}-model structure, if we choose for instance the functorial factorizations given by the small object argument).} In fact, this model structure is finitely generated, i.e. the domains and codomains of the maps in $\cS I$ and $\cS J$ are $n$-small $\tt DG\cD A$-objects, $n\in\N$, relative to $\op{Cof}$. Indeed, these sources and targets are $\cS D^k_\bullet$ ($k\ge 1$), $\cS S^k_\bullet$ ($k\ge 0$), and $\cO$. We already observed (see Theorem \ref{FinGenModDGDM}) that $D^k_\bullet$ ($k\ge 1$), $S^k_\bullet$ ($k\ge 0$), and $0$ are $n$-small $\tt DG\cD M$-objects with respect to all $\tt DG\cD M$-morphisms. If $\frak S_\bullet$ denotes any of the latter chain complexes, this means that the covariant Hom functor $\op{Hom}_{\tt DG\cD M}({\frak S}_\bullet,-)$ commutes with all $\tt DG\cD M$-colimits $\colim_{\zb<\zl}M_{\zb,\bullet}$ for all limit ordinals $\zl$. It therefore follows from the adjointness property (\ref{Adjoint}) and the equation (\ref{EnrichmentEq}) that, for any $\tt DG\cD A$-colimit $\colim_{\zb<\zl}A_{\zb,\bullet}$, we have $$\h_{\tt DG\cD A}(\cS {\frak S}_\bullet,\colim_{\zb<\zl}A_{\zb,\bullet})\simeq \h_{\tt DG\cD M}({\frak S}_\bullet,\op{For}(\colim_{\zb<\zl}A_{\zb,\bullet}))=$$ $$\h_{\tt DG\cD M}({\frak S}_\bullet,\colim_{\zb<\zl}\op{For}(A_{\zb,\bullet}))=\colim_{\zb<\zl}\h_{\tt DG\cD M}({\frak S}_\bullet,\op{For}(A_{\zb,\bullet}))\simeq$$ $$\colim_{\zb<\zl}\h_{\tt DG\cD A}(\cS {\frak S}_\bullet,A_{\zb,\bullet})\;.$$

\section{Description of $\tt DG\cD A$-cofibrations}

\subsection{Preliminaries}

The next lemma allows us to define non-split {\small RS$\cD$A}-s, as well as ${\tt DG\cD A}$-morphisms from such an {\small RS$\cD$A} into another differential graded $\cD$-algebra.

\begin{lem}\label{LemRSA} Let $(T,d_T)\in\tt DG\cD A$, let $(g_j)_{j\in J}$ be a family of symbols of degree $n_j\in \N$, and let $V=\bigoplus_{j\in J}\cD\cdot g_j$ be the free non-negatively graded $\cD$-module with homogeneous basis $(g_j)_{j\in J}$.\smallskip

(i) To endow the graded $\cD$-algebra $T\0\cS V$ with a differential graded $\cD$-algebra structure $d$, it suffices to define \be\label{CondRSADiff}d g_j\in T_{n_j-1}\cap d_T^{-1}\{0\}\;,\ee to extend $d$ as $\cD$-linear map to $V$, and to equip $T\0\cS V$ with the differential $d$ given, for any $t\in T_p,\,v_1\in V_{n_1},\,\ldots,\,v_k\in V_{n_k}\,$, by \be\label{DefRSADiff}d({t}\0 v_1\odot\ldots\odot v_k)=$$ $$d_T({t})\0 v_1\odot\ldots\odot v_k+(-1)^p\sum_{\ell=1}^k(-1)^{n_\ell\sum_{j<\ell}n_j}({t}\ast d(v_\ell))\0v_1\odot\ldots\widehat{\ell}\ldots\odot v_k\;,\ee where $\ast$ is the multiplication in $T$. If $J$ is a well-ordered set, the natural map $$(T,d_T)\ni {t}\mapsto {t}\0 1_\cO\in (T\boxtimes\cS V,d)$$ is a {\small RS$\cD\!$A}.\smallskip

(ii) Moreover, if $(B,d_B)\in{\tt DG\cD A}$ and $p\in{\tt DG\cD A}(T,B)$, it suffices -- to define a morphism $q\in{\tt DG\cD A}(T\boxtimes\cS V,B)$ (where the differential graded $\cD$-algebra $(T\boxtimes\cS V,d)$ is constructed as described in (i)) -- to define \be\label{CondRSAMorph}q(g_j)\in B_{n_j}\cap d_B^{-1}\{p\,d(g_j)\}\;,\ee to extend $q$ as $\cD$-linear map to $V$, and to define $q$ on $T\0\cS V$ by \be\label{DefRSAMorph}q({t}\0 v_1\odot\ldots\odot v_k)=p({t})\star q(v_1)\star\ldots\star q(v_k)\;,\ee where $\star$ denotes the multiplication in $B$.\end{lem}

The reader might consider that the definition of $d(t\0 f)$, $f\in\cO$, is not an edge case of Equation (\ref{DefRSADiff}); if so, it suffices to add the definition $d(t\0 f)=d_T(t)\0 f\,.$ Note also that Equation (\ref{DefRSADiff}) is the only possible one. Indeed, denote the multiplication in $T\0\cS V$ (see Equation (13)) by $\diamond$ and choose, to simplify, $k=2$. Then, if $d$ is any differential, which is compatible with the graded $\cD$-algebra structure of $T\0\cS V$, and which coincides with $d_T(t)\0 1_\cO\simeq d_T(t)$ on any $t\0 1_\cO\simeq t\in T$ (since $(T,d_T)\to (T\boxtimes\cS V,d)$ must be a $\tt DG\cD A$-morphism) and with $d(v)\0 1_\cO\simeq d(v)$ on any $1_T\0 v\simeq v\in V$ (since $d(v)\in T$), then we have necessarily
\bea & d(t\0 v_1\odot v_2)=\eea
\bea & d(t\0 1_\cO)\diamond (1_{T}\0 v_1)\diamond (1_{T}\0 v_2)+\\ &(-1)^p(t\0 1_\cO)\diamond d(1_{T}\0 v_1)\diamond (1_{T}\0 v_2)+\\ &(-1)^{p+n_1}(t\0 1_\cO)\diamond (1_{T}\0 v_1)\diamond d(1_{T}\0 v_2)=\eea

$$(d_T(t)\0 1_\cO)\diamond (1_{T}\0 v_1)\diamond (1_{T}\0 v_2)+$$ $$(-1)^p(t\0 1_\cO)\diamond (d(v_1)\0 1_\cO)\diamond (1_{T}\0 v_2)+$$ $$(-1)^{p+n_1}(t\0 1_\cO)\diamond (1_{T}\0 v_1)\diamond (d(v_2)\0 1_\cO)=$$

$$d_T(t)\0 v_1\odot v_2 + (-1)^p(t\ast d(v_1))\0 v_2+(-1)^{p+n_1n_2}(t\ast d(v_2))\0 v_1\;.$$
An analogous remark holds for Equation (\ref{DefRSAMorph}).

\begin{proof} It is easily checked that the {\small RHS} of Equation (\ref{DefRSADiff}) is graded symmetric in its arguments $v_i$ and $\cO$-linear with respect to all arguments. Hence, the map $d$ is a degree $-1$ $\cO$-linear map that is well-defined on $T\0\cS V$. To show that $d$ endows $T\0\cS V$ with a differential graded $\cD$-algebra structure, it remains to prove that $d$ squares to 0, is $\cD$-linear and is a graded derivation for $\diamond$. The last requirement follows immediately from the definition, for $\cD$-linearity it suffices to prove linearity with respect to the action of vector fields -- what is a straightforward verification, whereas 2-nilpotency  is a consequence of Condition (\ref{CondRSADiff}). The proof of (ii) is similar. \end{proof}

We are now prepared to give an example of a non-split {\small RS$\cD$A}.

\begin{ex}\label{Lem1}\emph{Consider the generating cofibrations $\iota_n:S^{n-1}\to D^n$, $n\ge 1$, and $\iota_0:0\to S^0$ of the model structure of $\tt DG\cD M$. The {\it pushouts} of the induced generating cofibrations $$\psi_n=\cS(\iota_n)\quad\text{and}\quad \psi_0=\cS(\iota_0)$$ of the transferred model structure on $\tt DG\cD A$ are important instances of {\small RS$\cD$A}-s -- see Figure 2 and Equations (\ref{kappa}), (\ref{RSA-d}), (\ref{i}), (\ref{Cond2}), and (\ref{j}).} \end{ex}

\begin{proof} We first consider a pushout diagram for $\psi:=\psi_n$, for $n\ge 1$: see Figure \ref{PDiag},
\begin{figure}[h]
\begin{center}
\begin{tikzpicture}
  \matrix (m) [matrix of math nodes, row sep=3em, column sep=3em]
    {  \cS(S^{n-1}) & (T,d_T)  \\
       \cS(D^n) & \\ };
 \path[->]
 (m-1-1) edge  node[above] {$\scriptstyle{\zf}$} (m-1-2);
  \path[->]
 (m-1-1) edge  node[left] {$\scriptstyle{\psi}$} (m-2-1);
\end{tikzpicture}
\end{center}
\caption{Pushout diagram}\label{PDiag}
\end{figure}
where $(T,d_T)\in \tt DG\mathcal{D}A$ and where $\phi:(\cS(S^{n-1}),0)\to(T,d_T)$ is a $\tt DG\cD A$-morphism. \medskip

In the following, the generator of $S^{n-1}$ (resp., the generators of $D^n$) will be denoted by $1_{n-1}$ (resp., by $\mathbb{I}_n$ and $s^{-1}\mathbb{I}_n$, where $s^{-1}$ is the desuspension operator).\medskip

Note that, since $\cS(S^{n-1})$ is the free {\small DG$\cD$A} over the {\small DG$\cD$M} $S^{n-1}$, the $\tt DG\cD A$-morphism $\phi$ is uniquely defined by the $\tt DG\mathcal{D}M$-morphism $\phi|_{S^{n-1}}: S^{n-1}\to \op{For}(T,d_T)$, where $\op{For}$ is the forgetful functor. Similarly, since $S^{n-1}$ is, as {\small G$\cD$M}, free over its generator $1_{n-1}$, the restriction $\phi|_{S^{n-1}}$ is, as $\tt G\cD M$-morphism, completely defined by its value $\phi(1_{n-1})\in T_{n-1}$. The map $\phi|_{S^{n-1}}$ is then a $\tt DG\cD M$-morphism if and only if we choose \be\label{kappa}\zk_{n-1}:=\phi(1_{n-1})\in\ker_{n-1}d_T\;.\ee

We now define the pushout of $(\psi,\phi)$: see Figure \ref{CPD}.
\begin{figure}[h]
\begin{center}
\begin{tikzpicture}
  \matrix (m) [matrix of math nodes, row sep=3em, column sep=3em]
    {  \cS(S^{n-1}) & (T,d_T)  \\
       \cS(D^n) & (T\boxtimes\cS(S^n),d)  \\ };
 \path[->]
 (m-1-2) edge  node[right] {$\scriptstyle{i}$} (m-2-2);
 \path[->]
 (m-1-1) edge  node[above] {$\scriptstyle{\zf}$} (m-1-2);
  \path[->]
 (m-1-1) edge  node[left] {$\scriptstyle{\psi}$} (m-2-1);
  \path[->]
 (m-2-1) edge  node[above] {$\scriptstyle{j}$} (m-2-2);
\end{tikzpicture}
\caption{Completed pushout diagram}\label{CPD}
\end{center}
\end{figure}
\noindent In the latter diagram, the differential $d$ of the {\small G$\cD$A} $T\boxtimes\cS(S^n)$ is defined as described in Lemma \ref{LemRSA}. Indeed, we deal here with the free non-negatively graded $\cD$-module $S^n=S^n_n=\cD\cdot 1_n$ and set $$d(1_n):=\zk_{n-1}=\zf(1_{n-1})\in\ker_{n-1}d_T\;.$$ Hence, if $x_\ell\cdot 1_n\in\cD\cdot 1_n$ (to simplify notation we denote in the following by $x_\ell$ both, the differential operator $x_\ell\in\cD$ and the element $x_\ell\cdot 1_n\in S^n$), we get $d(x_\ell)=x_\ell\cdot\zk_{n-1}$, and, if $t\in T_p$, we obtain \be\label{RSA-d}d({t}\0 x_1\odot\ldots\odot x_k)=$$ $$d_T({t})\0 x_1\odot\ldots\odot x_k+(-1)^p\sum_{\ell=1}^k(-1)^{n(\ell-1)}({t}\ast (x_\ell\cdot\zk_{n-1}))\0 x_1\odot\ldots\widehat{\ell}\ldots\odot x_k\;,\ee see Equation (\ref{DefRSADiff}). {Finally} the map \be\label{i}i:(T,d_T)\ni t\mapsto t\0 1_\cO\in (T\boxtimes\cS(S^n),d)\ee is a (non-split) {\small RS$\cD$A}{, see Definition \ref{RSullDAlg}}.\medskip

Just as $\phi$, the $\tt DG\cD A$-morphism $j$ is completely defined if we define it as $\tt DG\cD M$-morphism on $D^n$. The choices of $j(\mathbb{I}_n)$ and $j(s^{-1}\mathbb{I}_{n})$ define $j$ as $\tt G\cD M$-morphism. The commutation condition of $j$ with the differentials reads \be\label{Cond1}j(s^{-1}\mathbb{I}_n)=d\,j(\mathbb{I}_n)\;:\ee only $j(\mathbb{I}_n)$ can be chosen freely in $(T\0\cS(S^n))_n\,$.\medskip

The diagram of Figure \ref{CPD} is now fully described. To show that it commutes, observe that, since the involved maps $\phi,i,\psi$, and $j$ are all $\tt DG\cD A$-morphisms, it suffices to check commutation for the arguments $1_\cO$ and $1_{n-1}$. Since differential graded $\cD$-algebras are systematically assumed to be unital, only the second case is non-obvious. We get the condition \be\label{Cond2}d\,j(\mathbb{I}_n)=\zk_{n-1}\0 1_\cO\;.\ee It now suffices to set \be\label{j}j(\mathbb{I}_n)=1_T\0 1_n\in(T\0\cS(S^{n}))_n\;.\ee

To prove that the commuting diagram of Figure \ref{CPD} is the {searched for pushout}, it now suffices to prove its universality. Therefore, take $(B,d_B)\in\tt DG\cD A$, as well as two $\tt DG\cD A$-morphisms $i':(T,d_T)\to (B,d_B)$ and $j':\cS(D^n)\to(B,d_B)$, such that $j'\circ \psi=i'\circ\phi$, and show that there is a unique $\tt DG\mathcal{D}A$-morphism $\chi:(T\boxtimes \cS(S^n),d)\to (B,d_B)$, such that $\chi\circ i=i'$ and $\chi\circ j=j'$.\medskip

If $\chi$ exists, we have necessarily $$\chi(t\0 x_1\odot\ldots\odot x_k)=\chi((t\0 1_\cO)\diamond (1_T\0 x_1)\diamond \ldots\diamond (1_T\0 x_k))$$ \be\label{UP1}=\chi(i(t))\star \chi(1_T\0 x_1)\star\ldots\star \chi(1_T\0 x_k)\;,\ee where we used the same notation as above. Since any differential operator is generated by functions and vector fields, we get \be\label{UP2}\chi(1_T\0 x_i)=\chi(1_T\0 x_i\cdot 1_n)=x_i\cdot \chi(1_T\0 1_n)=x_i\cdot \chi(j(\mathbb{I}_n))=x_i\cdot j'(\mathbb{I}_n)=j'(x_i\cdot\mathbb{I}_n)\;.\ee When combining (\ref{UP1}) and (\ref{UP2}), we see that, if $\chi$ exists, it is necessarily defined by \be\label{UP3}\chi(t\0 x_1\odot\ldots\odot x_k)=i'(t)\star j'(x_1\cdot\mathbb{I}_n)\star\ldots\star j'(x_k\cdot\mathbb{I}_n)\;.\ee This solves the question of uniqueness.\medskip

We now convince ourselves that (\ref{UP3}) defines a $\tt DG\cD A$-morphism $\chi$ (let us mention explicitly that we set in particular $\chi(t\0 f)=f\cdot i'(t)$, if $f\in\cO$). It is straightforwardly verified that $\chi$ is a well-defined $\cD$-linear map of degree 0 from $T\0\cS(S^n)$ to $B$, which respects the multiplications and the units. The interesting point is the chain map property of $\chi$. Indeed, consider, to simplify, the argument $t\0 x$, what will disclose all relevant insights. Assume again that $t\in T_p$ and $x\in S^n$, and denote the differential of $\cS(D^n)$, just as its restriction to $D^n$, by $s^{-1}$. It follows that $$d_B(\chi(t\0 x))=i'(d_T(t))\star j'(x\cdot\mathbb{I}_n)+(-1)^{p}\,i'(t)\star j'(x\cdot s^{-1}\mathbb{I}_n)\;.$$ Since $\psi(1_{n-1})=s^{-1}\mathbb{I}_n$ and $j'\circ\psi=i'\circ\phi$, we obtain $j'(s^{-1}\mathbb{I}_n)=i'(\zf(1_{n-1}))=i'(\zk_{n-1})$. Hence, $$d_B(\chi(t\0 x))=\chi(d_T(t)\0 x)+(-1)^{p}\,i'(t)\star i'(x\cdot\zk_{n-1})=$$ $$\chi(d_T(t)\0 x+(-1)^{p}t\ast(x\cdot\zk_{n-1}))=\chi(d(t\0 x))\;.$$ As afore-mentioned, no new feature appears, if we replace $t\0 x$ by a general argument.\medskip

As the conditions $\chi\circ i=i'$ and $\chi\circ j=j'$ are easily checked, this completes the proof of the statement that any pushout of any $\psi_n$, $n\ge 1$, is {a {\small RS$\cD$A}}.\medskip

The proof of the similar claim for $\psi_0$ is analogous and even simpler, and will not be detailed here.\end{proof}

Actually pushouts of $\psi_0$ are border cases of pushouts of the $\psi_n$-s, $n\ge 1$. In other words, to obtain a pushout of $\psi_0$, it suffices to set, in Figure \ref{CPD} and in Equation (\ref{RSA-d}), the degree $n$ to 0. Since we consider exclusively non-negatively graded complexes, we then get $\cS(S^{-1})=\cS(0)=\cO$, $\cS(D^0)=\cS(S^0)$, and $\zk_{-1}=0$.

\subsection{$\tt DG\cD A$-cofibrations}

The following theorem characterizes the cofibrations of the cofibrantly generated model structure we constructed on $\tt DG\mathcal{D}A$.

\begin{theo}\label{Cof} The $\tt DG\mathcal{D}A$-cofibrations are exactly the retracts of the relative Sullivan $\cD$-algebras. \end{theo}

Since the $\tt DG\mathcal{D}A$-cofibrations are exactly the retracts of the transfinite compositions of pushouts of generating cofibrations $$\psi_n:\cS(S^{n-1})\to\cS(D^n),\quad n\ge 0\;,$$ the proof of Theorem \ref{Cof} reduces to the proof of

\begin{theo}\label{Reduction} The transfinite compositions of pushouts of $\psi_n$-s, $n\ge 0$, are exactly the relative Sullivan $\cD$-algebras.\end{theo}

\begin{lem}\label{DirSumTenPro} For any $M,N\in\tt DG\cD M$, we have $$\cS(M\oplus N)\simeq \cS M\0 \cS N\;$$ in $\tt DG\cD A\,.$\end{lem}

\begin{proof} 
It
suffices to remember that the binary coproduct in the category $\tt DG\cD M=Ch_+(\cD)$ (resp., the category $\tt DG\cD A=CMon(DG\cD M)$) of non-negatively graded chain complexes of $\cD$-modules (resp., of commutative monoids in $\tt DG\cD M$) is the direct sum (resp., the tensor product). The conclusion then follows from the facts that $\cS$ is the left adjoint of the forgetful functor and that any left adjoint commutes with colimits.\end{proof}

Any ordinal is zero, a successor ordinal, or a limit ordinal. We denote the class of all successor ordinals (resp., all limit ordinals) by $\frak{O}_s$ (resp., $\frak{O}_\ell$).

\begin{proof}[Proof of Theorem \ref{Reduction}] (i) Consider an ordinal $\zl$ and a $\zl$-sequence in $\tt DG\cD A$, i.e., a colimit respecting functor $X:\zl\to \tt DG\cD A$ (here $\zl$ is viewed as the category whose objects are the ordinals $\za<\zl$ and which contains a unique morphism $\za\to\zb$ if and only if $\za\le\zb$): $$X_0\to X_1\to \ldots \to X_n\to X_{n+1}\to\ldots X_\zw\to X_{\zw+1}\to \ldots \to X_\za\to X_{\za+1}\to \ldots$$ We assume that, for any $\za$ such that $\za+1<\zl$, the morphism $X_\za\to X_{\za+1}$ is a pushout of some $\psi_{n_{\za+1}}$ ($n_{\za+1}\ge 0$). Then the morphism $X_0\to \op{colim}_{\za<\zl}X_\za$ is exactly what we call a transfinite composition of pushouts of $\psi_n$-s. Our task is to show that this morphism is a {\small RS$\cD$A}.\medskip

We first compute the terms $X_\za$, $\za<\zl,$ of the $\zl$-sequence, then we determine its colimit. For $\za<\zl$ (resp., for $\za<\zl, \za\in{\frak O}_s$), we denote the differential graded $\cD$-algebra $X_\za$ (resp., the $\tt DG\cD A$-morphism $X_{\za-1}\to X_{\za}$) by $(A_\za,d_\za)$ (resp., by $X_{\za,\za-1}:(A_{\za-1},d_{\za-1})\to (A_{\za},d_{\za})$). Since $X_{\za,\za-1}$ is the pushout of some $\psi_{n_{\za}}$ along some $\tt DG\cD A$-morphism $\zf_{\za}$, its target algebra is of the form \be\label{Successor}(A_{\za},d_{\za})=(A_{\za-1}\boxtimes\cS\langle a_{\za}\ra,d_{\za})\;\ee and $X_{\za,\za-1}$ is the canonical inclusion \be\label{SuccessorMorph}X_{\za,\za-1}:(A_{\za-1},d_{\za-1})\ni \frak{a}_{\za-1}\mapsto \frak{a}_{\za-1}\0 1_\cO\in (A_{\za-1}\boxtimes\cS\langle a_\za\rangle,d_\za)\;,\ee see Example \ref{Lem1}. Here $a_{\za}$ is the generator $1_{n_{\za}}$ of $S^{n_{\za}}$ and $\langle a_{\za}\ra$ is the free non-negatively graded $\cD$-module $S^{n_{\za}}=\cD\cdot a_{\za}$ concentrated in degree $n_{\za}$; further, the differential \be\label{RSA-d-2}d_{\za}\;\;\text{is defined by (\ref{RSA-d}) from}\;\; d_{\za-1}\;\;\text{and}\;\;\zk_{n_{\za}-1}:=\zf_{\za}(1_{n_{\za}-1})\;.\ee In particular, $A_1=A_0\boxtimes\cS\langle a_1\ra\,,$ $d_1(a_1)=\zk_{n_1-1}=\zf_1(1_{n_1-1})\in A_0\,,$ and $X_{10}:A_0\to A_1$ is the inclusion.\medskip

\begin{lem}\label{Lem4} For any $\za<\zl$, we have
\be\label{Lem41}A_{\za}\simeq A_0\otimes \cS \langle a_\delta: \delta\leq\za, \zd\in\frak{O}_s\rangle\;\ee as a graded $\cD$-algebra,
and
\be\label{Lem42}d_{\za}(a_{\delta})\in A_0\otimes \cS \langle  a_\ze: \ze< \delta, \ze\in\frak{O}_s\rangle\;,\ee
for all $\delta\leq \za$, $\zd\in\frak{O}_s$. Moreover, for any $\zg\le\zb\le \za <\zl$, we have $$A_\zb=A_\zg\0\cS\langle a_\zd:\zg<\zd\le\zb,\zd\in{\frak O}_s\rangle$$ and the $\tt DG\cD A$-morphism $X_{\zb\zg}$ is the natural inclusion \be\label{Lem43}X_{\zb\zg}: (A_\zg,d_\zg)\ni \frak{a}_\zg\mapsto \frak{a}_\zg\0 1_\cO\in (A_\zb,d_\zb)\;.\ee Since the latter statement holds in particular for $\zg=0$ and $\zb=\za$, the $\tt DG\cD A$-inclusion $X_{\za 0}:(A_0,d_0)\to (A_\za,d_\za)$ is a {\small RS$\cD$A} $(\,$for the natural ordering of $\{a_\delta: \zd\le\za,\zd\in\frak{O}_s\}\,).$\end{lem}

\begin{proof}[Proof of Lemma \ref{Lem4}] To prove that this claim (i.e., Equations (\ref{Lem41}) -- (\ref{Lem43})) is valid for all ordinals that are smaller than $\zl$, we use a transfinite induction. Since the assertion obviously holds for $\za=1,$ it suffices to prove these properties for $\za<\zl$, assuming that they are true for all $\zb<\za$. We distinguish (as usually in transfinite induction) the cases $\za\in\frak{O}_s$ and $\za\in\frak{O}_\ell$.\medskip

If $\za\in\frak{O}_s$, it follows from Equation (\ref{Successor}), from the induction assumption, and from Lemma \ref{DirSumTenPro}, that $$A_\za=A_{\za-1}\0\cS\langle a_\za\ra\simeq A_0\0 \cS\langle a_\zd: \zd\le \za,\zd\in\frak{O}_s\ra\;,$$ as graded $\cD$-algebra. Further, in view of Equation (\ref{RSA-d-2}) and the induction hypothesis, we get $$d_\za(a_\za)=\zf_{\za}(1_{n_\za-1})\in A_{\za-1}=A_0\0 \cS\langle a_\zd: \zd<\za,\zd\in\frak{O}_s\ra\;,$$ and, for $\zd\le\za-1$, $\zd\in\frak{O}_s$, $$d_\za(a_\zd)=d_{\za-1}(a_\zd)\in A_0\otimes \cS \langle  a_\zg: \zg< \delta, \zg\in\frak{O}_s\rangle\;.$$ Finally, as concerns $X_{\zb\zg}$, the unique case to check is $\zg\le\za-1$ and $\zb=\za$. The $\tt DG\cD A$-map $X_{\za-1,\zg}$ is an inclusion $$X_{\za-1,\zg}: A_\zg\ni \frak{a}_\zg\mapsto \frak{a}_\zg\0 1_\cO\in A_{\za-1}\;$$ (by induction), and so is the $\tt DG\cD A$-map $$X_{\za,\za-1}:A_{\za-1}\ni \frak{a}_{\za-1}\mapsto \frak{a}_{\za-1}\0 1_\cO\in A_\za\;$$ (in view of (\ref{SuccessorMorph})). The composite $X_{\za\zg}$ is thus a $\tt DG\cD A$-inclusion as well.\medskip

In the case $\za\in\frak{O}_\ell$, i.e., $\za=\op{colim}_{\zb<\za}\zb$, we obtain $(A_\za,d_\za)=\op{colim}_{\zb<\za}(A_\zb,d_\zb)$ in $\tt DG\cD A$, since $X$ is a colimit respecting functor. The index set $\za$ is well-ordered, hence, it is a directed poset. Moreover, for any $\zd\le\zg\le\zb<\za$, the $\tt DG\cD A$-maps $X_{\zb\zd}$, $X_{\zg\zd}$, and $X_{\zb\zg}$ satisfy $X_{\zb\zd}=X_{\zb\zg}\circ X_{\zg\zd}\,$. It follows that the family $(A_\zb,d_\zb)_{\zb<\za},$ together with the family $X_{\zb\zg}$, $\zg\le\zb<\za$, is a direct system in $\tt DG\cD A$, whose morphisms are, in view of the induction assumption, natural inclusions $$X_{\zb\zg}:A_\zg\ni \frak{a}_\zg\mapsto \frak{a}_\zg\0 1_\cO\in A_\zb\;.$$ The colimit $(A_\za,d_\za)=\op{colim}_{\zb<\za}(A_\zb,d_\zb)$ is thus a direct limit. However, a direct limit in $\tt DG\cD A$ coincides with the corresponding direct limit in $\tt DG\cD M$, or even in $\tt Set$ (which is then naturally endowed with a differential graded $\cD$-algebra structure). As a set, the direct limit $(A_\za,d_\za)=\op{colim}_{\zb<\za}(A_\zb,d_\zb)$ is given by $$A_\za=\coprod_{\zb<\za}A_\zb/\sim\;,$$ where $\sim$ means that we identify $\frak{a}_\zg$, $\zg\le\zb$, with $$\frak{a}_\zg\sim X_{\zb\zg}(\frak{a}_\zg)=\frak{a}_\zg\0 1_\cO\;,$$ i.e., that we identify $A_\zg$ with $$A_\zg\sim A_\zg\0\cO\subset A_\zb\;.$$ It follows that $$A_\za=\bigcup_{\zb<\za}A_\zb=A_0\0\cS\langle a_\zd:\zd<\za,\zd\in\frak{O}_s\ra=A_0\0\cS\langle a_\zd:\zd\le\za,\zd\in\frak{O}_s\ra\;.$$ As just mentioned, this set $A_\za$ can naturally be endowed with a differential graded $\cD$-algebra structure. For instance, the differential $d_\za$ is defined in the obvious way from the differentials $d_\zb$, $\zb<\za$. In particular, any generator $a_\zd$, $\zd\le\za$, $\zd\in\frak{O}_s$, belongs to $A_\zd$. Hence, by definition of $d_\za$ and in view of the induction assumption, we get $$d_\za(a_\zd)=d_\zd(a_\zd)\in A_0\0\cS\langle a_\ze:\ze<\zd,\ze\in\frak{O}_s\ra\;.$$ {Finally}, since $X$ is colimit respecting, not only $A_\za =\colim_{\zb<\za}A_\zb=\bigcup_{\zb<\za}A_\zb$, but, furthermore, for any $\zg<\za$, the $\tt DG\cD A$-morphism $X_{\za\zg}:A_\zg\to A_\za$ is the map $X_{\za\zg}:A_\zg\to \bigcup_{\zb<\za}A_\zb$, i.e., the canonical inclusion.\end{proof}

We now come back to the proof of Part (i) of Theorem \ref{Reduction}, i.e., we now explain why the morphism $i:(A_0,d_0)\to C$, where $C=\op{colim}_{\za<\zl}(A_\za,d_\za)$ and where $i$ is the first of the morphisms that are part of the colimit construction, is a {\small RS$\cD$A} -- see above. If $\zl\in\frak{O}_s$, the colimit $C$ coincides with $(A_{\zl-1},d_{\zl-1})$ and $i=X_{\zl-1,0}$. Hence, the morphism $i$ is a {\small RS$\cD$A} in view of Lemma \ref{Lem4}. If $\zl\in\frak{O}_\ell$, the colimit $C=\op{colim}_{\za<\zl}(A_\za,d_\za)$ is, like above, the direct limit of the direct $\tt DG\cD A$-system $(X_\za=(A_\za,d_\za),X_{\za\zb})$ indexed by the directed poset $\zl$, whose morphisms $X_{\za\zb}$ are, in view of Lemma \ref{Lem4}, canonical inclusions. Hence, $C$ is again an ordinary union: \be\label{LimUnion}C=\bigcup_{\za<\zl}A_\za=A_0\0\cS\langle a_\zd: \zd<\zl,\zd\in\frak{O}_s\ra\;,\ee where the last equality is due to Lemma \ref{Lem4}. We define the differential $d_C$ on $C$ exactly as we defined the differential $d_\za$ on the direct limit in the proof of Lemma \ref{Lem4}. It is then straightforwardly checked that $i$ is a {\small RS$\cD$A}.\medskip

(ii) We still have to show that any {\small RS$\cD$A} $(A_0,d_0)\to (A_0\boxtimes\cS V,d)$ can be constructed as a transfinite composition of pushouts of generating cofibrations $\psi_n$, $n\ge 0$. Let $(a_j)_{j\in J}$ be the basis of the free non-negatively graded $\cD$-module $V$. Since $J$ is a well-ordered set, it is order-isomorphic to a unique ordinal $\zm=\{0,1,\ldots,n,\ldots,\zw,\zw+1,\ldots\}$, whose elements can thus be utilized to label the basis vectors. However, we prefer using the following order-respecting relabelling of these vectors: $$a_0\rightsquigarrow a_1, a_1\rightsquigarrow a_2,\ldots, a_n\rightsquigarrow a_{n+1},\ldots, a_\omega\rightsquigarrow a_{\omega+1}, a_{\omega+1}\rightsquigarrow a_{\omega+2},\ldots$$ In other words, the basis vectors of $V$ can be labelled by the successor ordinals that are strictly smaller than $\zl:=\zm+1\,$ (this is true, whether $\zm\in\frak{O}_s$, or $\zm\in\frak{O}_\ell\,$): $$V=\bigoplus_{\zd<\zl,\;\zd\in\frak{O}_s} \cD\cdot a_\zd\;.$$

For any $\za<\zl$, we now set $$(A_\za,d_\za):=(A_0\boxtimes \cS\langle a_\delta: \delta\leq\za, \zd\in\frak{O}_s\rangle,d|_{A_\za})\;.$$ It is clear that $A_\za$ is a graded $\cD$-subalgebra of $A_0\0\cS V$. Since $A_\za$ is generated, as an algebra, by the elements of the types $\frak{a}_0\0 1_\cO$ and $D\cdot(1_{A_0}\0 a_\zd)$, $D\in\cD$, $\zd\le\za,$ $\zd\in\frak{O}_s$, and since $$d(\frak{a}_0\0 1_\cO)=d_0(\frak{a}_0)\0 1_\cO\in A_\za$$ and $$d(D\cdot(1_{A_0}\0 a_\zd))\in A_0\0\cS\langle a_\ze:\ze<\zd,\ze\in\frak{O}_s\ra\subset A_\za\;,$$ the derivation $d$ stabilizes $A_\za$. Hence, $(A_\za,d_\za)=(A_\za,d|_{A_\za})$ is actually a differential graded $\cD$-subalgebra of $(A_0\boxtimes\cS V,d)$.\medskip

If $\zb\le\za<\zl$, the algebra $(A_\zb,d|_{A_\zb})$ is a differential graded $\cD$-subalgebra of $(A_\za,d|_{A_\za})$, so that the canonical inclusion $i_{\za\zb}:(A_\zb,d_\zb)\to(A_\za,d_\za)$ is a $\tt DG\cD A$-morphism. In view of the techniques used in (i), it is obvious that the functor $X=(A_-,d_-):\zl\to \tt DG\cD A$ respects colimits, and that the colimit of the whole $\zl$-sequence (remember that $\zl=\zm+1\in\frak{O}_s$) is the algebra $(A_\zm,d_\zm)=(A_0\boxtimes\cS V,d)$, i.e., the original algebra.\medskip

The {\small RS$\cD$A} $(A_0,d_0)\to (A_0\boxtimes\cS V,d)$ has thus been built as transfinite composition of canonical $\tt DG\cD A$-inclusions $i:(A_{\za},d_\za)\to (A_{\za+1},d_{\za+1})$, $\za+1<\zl$. Recall that $$A_{\za+1}=A_\za\0\cS\langle a_{\za+1}\ra\simeq A_\za\0\cS(S^n)\;,$$ if we set $n:=\deg(a_{\za+1})$. It suffices to show that $i$ is a pushout of $\psi_n$, see Figure \ref{CPD2}.
\begin{figure}[h]
\begin{center}
\begin{tikzpicture}
  \matrix (m) [matrix of math nodes, row sep=3em, column sep=3em]
    {  \cS(S^{n-1}) & (A_\za,d_\za)  \\
       \cS(D^n) & (A_\za\boxtimes\cS(S^n),d_{\za+1})  \\ };
 \path[->]
 (m-1-2) edge  node[right] {$\scriptstyle{i}$} (m-2-2);
 \path[->]
 (m-1-1) edge  node[above] {$\scriptstyle{\zf}$} (m-1-2);
  \path[->]
 (m-1-1) edge  node[left] {$\scriptstyle{\psi_n}$} (m-2-1);
  \path[->]
 (m-2-1) edge  node[above] {$\scriptstyle{j}$} (m-2-2);
\end{tikzpicture}
\caption{$i$ as pushout of $\psi_n$}\label{CPD2}
\end{center}
\end{figure}
\noindent We will detail the case $n\ge 1$. Since all the differentials are restrictions of $d$, we have $\zk_{n-1}:=d_{\za+1}(a_{\za+1})\in A_\za\cap\ker_{n-1}d_{\za}$, and $\zf(1_{n-1}):=\zk_{n-1}$ defines a $\tt DG\cD A$-morphism $\zf$, see Example \ref{Lem1}. When using the construction described in Example \ref{Lem1}, we get the pushout $i:(A_\za,d_\za)\to (A_\za\boxtimes\cS(S^n),\p)$ of $\psi_n$ along $\zf$. Here $i$ is the usual canonical inclusion and $\p$ is the differential defined by Equation (\ref{RSA-d}). It thus suffices to check that $\p=d_{\za+1}$. Let $\frak{a}_\za\in A^p_\za$ and let $x_1\simeq x_1\cdot\, a_{\za+1},\ldots,x_k\simeq x_k\cdot\, a_{\za+1}\in\cD\cdot\, a_{\za+1}=S^n$. Assume, to simplify, that $k=2$; the general case is similar. When denoting the multiplication in $A_\za$ (resp., $A_{\za+1}=A_\za\0\cS(S^n)$) as usual by $\ast$ (resp., $\star\,$), we obtain $$\p(\frak{a}_\za\0 x_1\odot x_2)=$$
$$d_\za(\frak{a}_\za)\0 x_1\odot x_2 + (-1)^p(\frak{a}_\za\ast (x_1\cdot \zk_{n-1}))\0 x_2+(-1)^{p+n}(\frak{a}_\za\ast (x_2\cdot\zk_{n-1}))\0 x_1=$$

$$(d_\za(\frak{a}_\za)\0 1_\cO)\star (1_{A_\za}\0 x_1)\star (1_{A_\za}\0 x_2)+$$ $$(-1)^p(\frak{a}_\za\0 1_\cO)\star ((x_1\cdot\zk_{n-1})\0 1_\cO)\star (1_{A_\za}\0 x_2)+$$ $$(-1)^{p+n}(\frak{a}_\za\0 1_\cO)\star (1_{A_\za}\0 x_1)\star ((x_2\cdot\zk_{n-1})\0 1_\cO)=$$

$$d_{\za+1}(\frak{a}_\za\0 1_\cO)\star (1_{A_\za}\0 x_1)\star (1_{A_\za}\0 x_2)+$$ $$(-1)^p(\frak{a}_\za\0 1_\cO)\star d_{\za+1}(1_{A_\za}\0 x_1)\star (1_{A_\za}\0 x_1)+$$ $$(-1)^{p+n}(\frak{a}_\za\0 1_\cO)\star (1_{A_\za}\0 x_1)\star d_{\za+1}(1_{A_\za}\0 x_2)=$$

$$d_{\za+1}(\frak{a}_\za\0 x_1\odot x_2)\;.$$
\end{proof}

\section{Explicit functorial factorizations}\label{Factorizations}

The main idea of Subsection \ref{Condition2} is the decomposition of an arbitrary $\tt DG\cD A$-morphism $\zf:A\to B$ into a weak equivalence $i:A\to A\0\cS U$ and a fibration $p:A\0\cS U\to B$. It is easily seen that $i$ is a split relative Sullivan $\cD$-algebra. Indeed, \be\label{U}U=P(B)=\bigoplus_{n>0}\bigoplus_{b_n\in B_n}D^n_\bullet\in\tt DG\cD M\ee with differential $d_U=d_P$ defined by \be\label{dU}d_U(s^{-1}\mathbb{I}_{b_n})=0\quad\text{and}\quad d_U(\mathbb{I}_{b_n})=s^{-1}\mathbb{I}_{b_n}\;.\ee Hence, $\cS U\in\tt DG\cD A$, with differential $d_S$ induced by $d_U$, and $A\0 \cS U\in\tt DG\cD A$, with differential \be\label{d}d_1=d_A\0 \id +\id\0 d_S\;.\ee Therefore, $i:A\to A\0\cS U$ is a $\tt DG\cD A$-morphism. Since $U$ is the free non-negatively graded $\cD$-module with homogeneous basis $$G=\{s^{-1}\mathbb{I}_{b_n}, \mathbb{I}_{b_n}:b_n\in B_n, n>0\}\;,$$ all the requirements of the definition of a split {\small RS$\cD$A} are obviously satisfied, except that we still have to check the well-ordering and the lowering condition.\medskip

Since every set can be well-ordered, we first choose a well-ordering on each $B_n$, $n>0$: if $\zl_n$ denotes the unique ordinal that belongs to the same equivalence class of well-ordered sets, the elements of $B_n$ can be labelled by the elements of $\zl_n$. Then we define the following total order: the $s^{-1}\mathbb{I}_{b_1}$, $b_1\in B_1$, are smaller than the $\mathbb{I}_{b_1}$, which are smaller than the $s^{-1}\mathbb{I}_{b_2}$, and so on ad infinitum. The construction of an infinite decreasing sequence in this totally ordered set amounts to extracting an infinite decreasing sequence from a finite number of ordinals $\zl_1,\zl_1,\ldots,\zl_k$. Since this is impossible, the considered total order is a well-ordering. The lowering condition is thus a direct consequence of Equations (\ref{dU}) and (\ref{d}).\medskip

Let now $\{\zg_\za:\za \in J\}$ be the set $G$ of generators endowed with the just defined well-order. Observe that, if the label $\za$ of the generator $\zg_\za$ increases, its degree $\deg\zg_\za$ increases as well, i.e., that \be\label{Mini}\za\le\zb\quad \Rightarrow\quad \deg \zg_\za\le\deg\zg_\zb\;.\ee

\def\fb{{\mathfrak b}}

{Finally}, any $\tt DG\cD A$-morphism $\zf:A\to B$ admits a functorial factorization \be\label{TrivCofFib}A\stackrel{i}{\longrightarrow}A\0\cS U\stackrel{p}{\longrightarrow}B\;,\ee where $p$ is a fibration and $i$ is a weak equivalence, as well as a split {\small RS$\cD$A}. In view of Theorem \ref{Cof}, the morphism $i$ is thus a cofibration, with the result that we actually constructed a natural decomposition $\zf=p\circ i$ of an arbitrary $\tt DG\cD A$-morphism $\zf$ into $i\in\text{\small TrivCof}$ and $p\in\text{\small Fib}$. The description of this factorization is summarized below, in Theorem \ref{P:c-tf_tc-f}, which provides essentially an explicit natural `{\small Cof -- TrivFib}' decomposition \be\label{CofTrivFib}A\stackrel{i'}{\longrightarrow}A\0\cS U'\stackrel{p'}{\longrightarrow}B\;.\ee

Before stating Theorem \ref{P:c-tf_tc-f}, we sketch the construction of the factorization (\ref{CofTrivFib}). To simplify, we denote algebras of the type $A\0 \cS V_k$ by $R_{V_k}$, or simply $R_k\,$.\medskip

We start from the `small' `{\small Cof -- Fib}' decomposition (\ref{TrivCofFib}) of a $\tt DG\cD A$-morphism $A\stackrel{\zf}{\longrightarrow} B$, i.e., from the factorization $A\stackrel{i}{\longrightarrow}R_U\stackrel{p}{\longrightarrow}B$. To find a substitute $q$ for $p$, which is a trivial fibration, {\it we mimic an idea used in the construction of the Koszul-Tate resolution: we add generators to improve homological properties}.\medskip

Note first that $H(p)$ is surjective if, for any homology class $[\zb_n]\in H_n(B)$, there is a class $[\zr_n]\in H_n(R_U)$, such that $[p\,\zr_n]=[\zb_n]$. Hence, consider all the homology classes $[\zb_n]$, $n\ge 0,$ of $B$, choose in each class a representative $\dot\zb_n\simeq [\zb_n]$, and add generators $\mbi_{\dot\zb_n}$ to those of $U$. It then suffices to extend the differential $d_1$ (resp., the fibration $p$) defined on $R_U=A\0\cS U$, so that the differential of $\mbi_{\dot\zb_n}$ vanishes (resp., so that the projection of $\mbi_{\dot\zb_n}$ coincides with $\dot\zb_n$) ($\rhd_1$ -- this triangle is just a mark that allows us to retrieve this place later on). To get a {\it functorial} `{\small Cof -- TrivFib}' factorization, we do not add a new generator $\mbi_{\dot\zb_n}$, for each homology class $\dot\zb_n\simeq [\zb_n]\in H_n(B)$, $n\ge 0,$ but we add a new generator $\mbi_{\zb_n}$, for each cycle $\zb_n\in\ker_n d_B$, $n\ge 0\,.$ Let us implement this idea in a rigorous manner. Assign the degree $n$ to $\mbi_{\zb_n}$ and set $$V_0:=U\oplus G_0:= U\oplus \langle \mbi_{\zb_n}: \zb_n\in\ker_n d_B, n\ge 0\ra=$$ \be\label{Newg}\langle s^{-1}\mbi_{b_n}, \mbi_{b_n}, \mbi_{\zb_n}: b_n\in B_n, n>0, \zb_n\in \ker_n d_B, n\ge 0 \ra\;.\ee Set now \be\label{Diffg}\zd_{V_0}(s^{-1}\mbi_{b_n})=d_1(s^{-1}\mbi_{b_n})=0,\;\;\zd_{V_0}\mbi_{b_n}=d_1\mbi_{b_n}={s^{-1}\mbi_{b_n}},\;\;\zd_{V_0}\mbi_{\zb_n}=0\;,\ee thus defining, in view of Lemma \ref{DiffGen}, a differential graded $\cD$-module structure on $V_0$. It follows that $(\cS V_0,\zd_{V_0})\in \tt DG\cD A$ and that \be\label{NewgTot}(R_0,\zd_0):=(A\0\cS V_0,d_A\0\id+\id\0\,\zd_{V_0})\in\tt DG\cD A\;.\ee Similarly, we set \be\label{Morpg}q_{V_0}(s^{-1}\mbi_{b_n})=p(s^{-1}\mbi_{b_n})=\ze(s^{-1}\mbi_{b_n})=d_Bb_n,\;\;q_{V_0}\mbi_{b_n}=p\mbi_{b_n}=\ze\mbi_{b_n}=b_n,\;\;q_{V_0}\mbi_{\zb_n}={\zb}_n\;.\ee We thus obtain, see Lemma \ref{MorpGen}, a morphism $q_{V_0}\in{\tt DG\cD M}(V_0,B)$ -- which uniquely extends to a morphism $q_{V_0}\in{\tt {DG\cD A}}(\cS V_0,B)$. Finally, \be\label{Morp}q_0=\zm_B\circ(\zf\0 q_{V_0})\in {\tt DG\cD A}(R_0,B)\;,\ee where $\zm_B$ denotes the multiplication in $B$. Let us emphasize that $R_U=A\0\cS U$ is a direct summand of $R_0=A\0\cS V_0$, and that $\zd_0$ and $q_0$ just extend the corresponding morphisms on $R_U$: $\zd_0|_{R_U}=d_1$ and $q_0|_{R_U}=p\,$.\medskip

So far we ensured that $H(q_0):H(R_0)\to H(B)$ is surjective; however, it must be injective as well, i.e., for any $\zs_n\in\ker \zd_0$, $n\ge 0,$ such that $H(q_0)[\zs_n]=0$, i.e., such that $q_0\zs_n\in\im d_B$, there should exist $\zs_{n+1}\in R_0$ such that \be\label{Extension}\zs_n=\zd_0\zs_{n+1}\;.\ee We denote by $\cB_0$ the set of $\zd_0$-cycles that are sent to $d_B$-boundaries by $q_0\,$: $${\cB}_0=\{\zs_n\in\ker\zd_0: q_0\zs_n\in\im d_B, n\ge 0\}\;.$$ In principle it now suffices to add, to the generators of $V_0$, generators $\mbi^1_{\zs_n}$ of degree $n+1$, $\zs_n\in{\cB}_0$, and to extend the differential $\zd_0$ on $R_0$ so that the differential of $\mbi^1_{\zs_n}$ coincides with $\zs_n$ ($\rhd_2$). However, it turns out that to obtain a {\it functorial} `{\small Cof -- TrivFib}' decomposition, we must add a new generator $\mbi^1_{\zs_n,\fb_{n+1}}$ of degree $n+1$, for each pair $(\zs_n,\fb_{n+1})$ such that $\zs_n\in\ker\zd_0$ and $q_0\zs_n=d_B\fb_{n+1}\,$: we set \be\label{Critset}{\frak B}_0=\{(\zs_n,\fb_{n+1}):\zs_n\in\ker\zd_0,\fb_{n+1}\in d_B^{-1}\{q_0\zs_n\},n\ge 0\}\ee and \be\label{Newgen}V_1:=V_0\oplus G_1:= V_0\oplus\langle \mbi^1_{\zs_n,\fb_{n+1}}:(\zs_n,\fb_{n+1})\in{\frak B}_0\ra\;.\ee To endow the graded $\cD$-algebra \be\label{VewDGDA}R_1:=A\0\cS V_1\simeq R_0\0\cS G_1\ee with a differential graded $\cD$-algebra structure $\zd_1$, we apply Lemma \ref{LemRSA}, with \be\label{DiffNewGen}\zd_1(\mbi^1_{\zs_n,\fb_{n+1}})=\zs_n\in (R_0)_n\cap \ker\zd_0\;,\ee exactly as suggested by Equation (\ref{Extension}). The differential $\zd_1$ is then given by Equation (\ref{DefRSADiff}) and it extends the differential $\zd_0$ on $R_0$. The extension of the $\tt DG\cD A$-morphism $q_0:R_0\to B$ by a $\tt DG\cD A$-morphism $q_1:R_1\to B$ is built from its definition \be\label{MorpNewGen}q_1(\mbi^1_{\zs_n,\fb_{n+1}})=\fb_{n+1}\in B_{n+1}\cap d_B^{-1}\{q_0\zd_1(\mbi^1_{\zs_n,\fb_{n+1}})\}\ee on the generators and from Equation (\ref{DefRSAMorph}) in Lemma \ref{LemRSA}.\medskip

{Finally}, starting from $(R_U,d_1)\in{\tt DG\cD A}$ and $p\in{\tt DG\cD A}(R_U,B)$, we end up -- when trying to make $H(p)$ bijective -- with $(R_{1},\zd_1)\in{\tt DG\cD A}$ and $q_1\in{\tt DG\cD A}(R_{1},B)$ -- so that the question is whether $H(q_1):H(R_1)\to H(B)$ is bijective or not. Since $(R_1,\zd_1)$ extends $(R_0,\zd_0)$ and $H(q_0):H(R_0)\to H(B)$ is surjective, it is easily checked that this property holds a fortiori for $H(q_1)$. However, when working with $R_1\supset R_0$, the `critical set' ${\cal B}_1\supset {\cal B}_0$ increases, so that we must add new generators $\mbi_{\zs_n}^2$, $\zs_n\in{\cal B}_1\setminus{\cal B}_0$, where $${\cal B}_1=\{\zs_n\in\ker\zd_1:q_1\zs_n\in\im d_B, n\ge 0\}\;.\quad (\rhd_3)$$ To build a {\it functorial$\,$} factorization, we consider not only the `critical set' \be\label{CritSet}{\frak B}_1=\{(\zs_n,\fb_{n+1}):\zs_n\in\ker\zd_1, \fb_{n+1}\in d_B^{-1}\{q_1\zs_n\},n\ge 0\}\;,\ee but also the module of new generators \be\label{NewGen}G_2=\langle\mbi^2_{\zs_n,\fb_{n+1}}:(\zs_n,\fb_{n+1})\in{\frak B}_1\ra\;,\ee indexed, not by ${\frak B}_1\setminus{\frak B}_0$, but by ${\frak B}_1$. Hence an iteration of the procedure (\ref{Critset}) - (\ref{MorpNewGen}) and the definition of a sequence $$(R_0,\zd_0)\rightarrow (R_1,\zd_1)\rightarrow(R_2,\zd_2)\rightarrow\ldots\rightarrow(R_{k-1},\zd_{k-1})\rightarrow (R_{k},\zd_{k})\rightarrow\ldots$$ of canonical inclusions of differential graded $\cD$-algebras $(R_k,\zd_k)$, $R_k=A\0\cS V_k$, $\zd_k|_{R_{k-1}}=\zd_{k-1}$, together with a sequence of ${\tt DG\cD A}$-morphisms $q_k:R_k\to B$, such that $q_k|_{R_{k-1}}=q_{k-1}$. The definitions of the differentials $\zd_k$ and the morphisms $q_k$ are obtained inductively, and are based on Lemma \ref{LemRSA}, as well as on equations of the same type as (\ref{DiffNewGen}) and (\ref{MorpNewGen}).\medskip

The direct limit of this sequence is a differential graded $\cD$-algebra $(R_V,d_2)=(A\0\cS V,d_2)$, together with a morphism $q:A\0\cS V\to B$.

As a set, the colimit of the considered system of canonically included algebras $(R_k,\zd_k)$, is just the union of the sets $R_k$, see Equation (\ref{LimUnion}). We proved above that this set-theoretical inductive limit can be endowed in the standard manner with a differential graded $\cD$-algebra structure and that the resulting algebra {\it is} the direct limit in $\tt DG\cD A$. One thus obtains in particular that $d_2|_{R_k}=\zd_k\,$.

Finally, the morphism $q:R_V\to B$ comes from the universality property of the colimit and it allows us to factor the morphisms $q_k:R_k\to B$ through $R_V$. We have: $q|_{R_k}=q_k\,$.\medskip

We will show that this morphism $A\0\cS V\stackrel{q}{\longrightarrow} B$ really leads to a `{\small Cof -- TrivFib'} decomposition $A\stackrel{j}{\longrightarrow} A\0\cS V\stackrel{q}{\longrightarrow} B$ of $A\stackrel{\zf}{\longrightarrow} B$.

\begin{theo}\label{P:c-tf_tc-f}
In $\tt DG\mathcal{D}A$, a functorial `{\small TrivCof -- Fib}' factorization $(i,p)$ and a functorial `{Cof -- TrivFib}' factorization $(j,q)$ of an arbitrary morphism
$$
\zf:(A,d_A)\to (B,d_B)\;,
$$
see Figure \ref{Fact}, can be constructed as follows:\medskip

\begin{figure}[h]
\begin{center}
\begin{tikzpicture}
  \matrix (m) [matrix of math nodes, row sep=3em, column sep=3em]
    {  (A,d_A) & (A\boxtimes\cS U,d_1)  \\
       (A\boxtimes\cS V,d_2) & (B,d_B)  \\ };
 \path[->>]
 (m-1-2) edge  node[right] {$\scriptstyle{p}$} (m-2-2);
 \path[->]
 (m-1-1) edge  node[below] {$\scriptstyle{\zf}$} (m-2-2);
 \path[>->]
 (m-1-1) edge  node[above] {$\sim$} node[below] {$\scriptstyle{i}$} (m-1-2);
  \path[>->]
 (m-1-1) edge  node[left] {$\scriptstyle{j}$} (m-2-1);
  \path[->>]
 (m-2-1) edge  node[above] {$\sim$} node[below] {$\scriptstyle{q}$} (m-2-2);
\end{tikzpicture}
\caption{Functorial factorizations}\label{Fact}
\end{center}
\end{figure}

(1) The module $U$ is the free non-negatively graded $\cD$-module with homogeneous basis $$\bigcup\,\{s^{-1}\mathbb{I}_{b_n},\mathbb{I}_{b_n}\}\;,$$ where the union is over all $b_n\in B_n$ and all $n>0$, and where $\deg({s^{-1}\mathbb{I}_{b_n}})=n-1$ and $\deg(\mathbb{I}_{b_n})=n\,.$ In other words, the module $U$ is a direct sum of copies of the discs $$D^n=\cD \cdot \mbi_{b_n}\oplus \cD \cdot s^{-1}\mbi_{b_n}\;,$$ $n>0$. The differentials $$s^{-1}:D^n\ni \mathbb{I}_{b_n}\to s^{-1}\mathbb{I}_{b_n}\in D^n$$ induce a differential $d_U$ in $U$, which in turn implements a differential $d_S$ in $\cS U$. The differential $d_1$ is then given by $d_1=d_A\0\id+\id\0 d_S\,.$ The trivial cofibration $i:A\to A\0\cS U$ is a split {\small RS$\cD\!$A} defined by $i:\frak{a}\mapsto\frak{a}\0 1_\cO$, and the fibration $p:A\0\cS U\to B$ is defined by $p=\zm_B\circ (\zf\0 \ze)$, where $\zm_B$ is the multiplication of $B$ and where $\ze(\mbi_{b_n})=b_n$ and $\ze(s^{-1}\mbi_{b_n})=d_Bb_n\,$.\medskip

(2) The module $V$ is the free non-negatively graded $\cD$-module with homogeneous basis $$\bigcup\,\{s^{-1}\mathbb{I}_{b_n},\mathbb{I}_{b_n},\mbi_{\zb_n},\mbi^1_{\zs_n,\fb_{n+1}},\mbi^2_{\zs_n,\fb_{n+1}},\ldots,\mbi^k_{\zs_n,\fb_{n+1}},\ldots\}\;,$$ where the union is over all $b_n\in B_n$, $n>0,$ all $\zb_n\in\ker_nd_B$, $n\ge 0$, and all pairs $$(\zs_n,\fb_{n+1}),\; n\ge 0,\;\, \text{in}\;\, {\frak B}_0,{\frak B}_1,\ldots, {\frak B}_k, \ldots,\;$$ respectively. The sequence of sets $${\frak B}_{k-1}=\{(\zs_n,\fb_{n+1}):\zs_n\in\ker\zd_{k-1}, \fb_{n+1}\in d_B^{-1}\{q_{k-1}\zs_n\},n\ge 0\}$$ is defined inductively, together with an increasing sequence of differential graded $\cD$-algebras $(A\0\cS V_k,\zd_k)$ and a sequence of morphisms $q_k:A\0\cS V_k\to B$, by means of formulas of the type (\ref{Critset}) - (\ref{MorpNewGen}) (see also (\ref{Newg}) - (\ref{Morp})). The degrees of the generators of $V$ are \be\label{GenDeg}n-1,\,n,\,n,\,n+1,\,n+1,\ldots, n+1,\ldots\ee The differential graded $\cD$-algebra $(A\0\cS V,d_2)$ is the colimit of the preceding increasing sequence of algebras: \be\label{DefD2}d_2|_{A\0\cS V_k}=\zd_k\;.\ee The trivial fibration $q:A\0\cS V\to B$ is induced by the $q_k$-s via universality of the colimit: \be\label{DefQ}q|_{A\0\cS V_k}=q_k\;.\ee {Finally}, the cofibration $j:A\to A\0\cS V$ is a (non-split) {\small RS$\cD\!$A}, which is defined as in (1) as the canonical inclusion; the canonical inclusion $j_k:A\to A\0\cS V_k\,$, $k>0\,$, is also a (non-split) {\small RS$\cD\!$A}, whereas $j_0:A\to A\0\cS V_0$ is a split {\small RS$\cD\!$A}.
\end{theo}

\begin{proof} See Appendix \ref{CofTrivFibProof}. \end{proof}

\begin{rem}\label{NonFuncFact}
{\em \begin{itemize}
\item If we are content with a non-functorial `{\small Cof -- TrivFib}' factorization, we may consider the colimit $A\otimes\cS {\cal V}$ of the sequence $A\otimes\cS {\cal V}_k$ that is obtained by adding only generators (see ($\rhd_1$)) $$\mbi_{\dot\zb_n},\;\, n\ge 0,\;\, \dot\zb_n\simeq[\zb_n]\in H_n(B)\;,$$ and by adding only generators (see ($\rhd_2$) and ($\rhd_3$)) $$\mbi_{\zs_n}^1,\mbi_{\zs_n}^2,\ldots,\;\, n\ge 0,\;\, \zs_n\in {\cal B}_0,{\cal B}_1\setminus{\cal B}_0,\ldots\;$$
\item An explicit description of the functorial fibrant and cofibrant replacement functors, induced by the `{\small TrivCof -- Fib}' and `{\small Cof -- TrivFib}' decompositions of Theorem \ref{P:c-tf_tc-f}, can be found in Appendix \ref{FunctReplFunct}.
\end{itemize}}
\end{rem}

\section{First remarks on Koszul-Tate resolutions}

In this last section, we {provide a first} insight into Koszul-Tate resolutions. {The Koszul-Tate resolution ({\small KTR}), which is used in Mathematical Physics and more precisely in \cite{Bar}, relies on horizontal differential operators, whose coordinate expression contains total derivatives. For instance, in the case of a unique base coordinate $t$, the total derivative with respect to $t$ is $$D_t=\p_t+\dot{q}\p_q+\ddot q\p_{\dot q}+\ldots\;,$$ where $q,\dot q,\ddot q,\ldots$ are infinite jet bundle fiber coordinates. The main concept of the jet bundle formalism is the Cartan connection $\cal C$, which allows to lift base differential operators $\p_t$ acting on base functions $\cO=\cO(X)$ to horizontal differential operators $D_t$ acting on the functions $\cO(J^\infty E)$ of the infinite jet bundle $J^\infty E\to X$ of a vector bundle $E\to X$. 
Hence, the total derivative $D_t^kF$ of a jet bundle function $F$ can be viewed as the action $\p_t^k\cdot F$ on $F$ of the corresponding base-derivative. In other words, one defines this action as the natural action by the corresponding lifted operator. Jet bundle functions $\cO(J^\infty E)$ thus become an algebra $$\cO(J^\infty E)\in\tt \cD A$$ over $\cD=\cD(X)$. In our algebraic geometric setting, there exists an infinite jet bundle functor $\cJ^\infty:{\tt \cO A}\to {\tt \cD A}$, which transforms the algebra $\cO(E)\in\cO\tt A$ of vector bundle functions into an algebra $$\cJ^\infty(\cO(E))\in\cD\tt A\;.$$ The latter is the algebraic geometric counterpart of the $\cD$-algebra $\cO(J^\infty E)$ used in smooth geometry. Recall now that (the prolongation of) a partial differential equation on the sections of $E$ can be viewed as an algebraic equation on the points of $J^\infty E$. The solutions of the latter form the critical surface $\zS\subset J^\infty E$. The function algebra $$\cO(J^\infty E)/I(\zS)\in\tt\cD A$$ of this stationary surface $\zS$ is the quotient of the $\cD$-algebra $\cO(J^\infty E)$ by the $\cD$-ideal $I(\zS)$ of those jet bundle functions that vanish on $\zS$. The Koszul-Tate resolution resolves this quotient. Now, for any $\cD$-ideal $\cI$, we think about $$\cJ^\infty(\cO(E))/\cI\in\tt \cD A$$ as the function algebra of some critical locus $\zS$. In our model categorical context, its (Koszul-Tate) resolution should be related to a cofibrant replacement of $\cJ^\infty(\cO(E))/\cI\in\tt \cD A$ in the model category $\tt DG\cD A$. This will be explained in detail below. Let us stress again, before proceeding, that in the present model categorical setting, the algebra $\cD$ is the algebra $\cD_X(X)$ of global sections of the sheaf $\cD_X$, where $X$ is a smooth affine algebraic variety.}\medskip

{In a separate paper \cite{PP}, we will give a new, general, and precise definition of Koszul-Tate resolutions. Instead of defining, as in homological algebra, a {\small KTR} for the quotient of some type of ring by an ideal, we will consider a (sheaf of) $\cD_X$-algebra(s) $\cA$ over an arbitrary smooth scheme $X$, as well as a differential graded $\cD_X$-algebra (sheaf) morphism $\zf:\cA\to\cB$. We will denote by $\cA[\cD_X]$ the ring of differential operators on $X$ with coefficients in $\cA$ and will define a $\cD$-geometric {\small KTR} of $\zf$ as a differential graded $\cA[\cD_X]$-algebra morphism $\psi:\cC\to\cB$, whose definition mimics the essential characteristics of our model categorical or cofibrant replacement {\small KTR} here above. It will turn out that such a {\small KTR} does always exist. We will further show that the {\small KTR} of a quotient ring \cite{Tate}, the {\small KTR} used in Mathematical Physics \cite{HT}, the {\small KTR} implemented by a compatibility complex \cite{Ver}, as well as our model categorical {\small KTR}, are all $\cD$-geometric Koszul-Tate resolutions. We will actually give precise comparison results for these Koszul-Tate resolutions, thus providing a kind of dictionary between different fields of science and their specific languages.}\medskip

Hence, the present section should be viewed as an introduction to topics on which we will elaborate in \cite{PP}.

\subsection{Undercategories of model categories}

When recalling that the coproduct in $\tt DG\cD A$ is the tensor product, we get from \cite{Hir2} that:

\begin{prop} For any differential graded $\cD$-algebra $A$, the coslice category $A\downarrow \tt DG\cD A$ carries a cofibrantly generated model structure given by the adjoint pair $L_{\0}:{\tt DG\cD A}\rightleftarrows A\downarrow{\tt DG\cD A}:\op{For}$, in the sense that its distinguished morphism classes are defined by $\op{For}$ and its generating cofibrations and generating trivial cofibrations are given by $L_\0\,$.\end{prop}

\subsection{Basics of jet bundle formalism}

The jet bundle formalism allows for a coordinate-free approach to partial differential equations ({\small PDE}-s), i.e., to (not necessarily linear) differential operators ({\small DO}-s) acting between sections of smooth vector bundles (the confinement to vector bundles does not appear in more advanced approaches). To uncover the main ideas, we implicitly consider in this subsection trivialized line bundles $E$ over a 1-dimensional manifold $X$, i.e., we assume that $E\simeq \R\times\R$.\medskip

The key-aspect of the jet bundle approach to {\small PDE}-s is the passage to purely algebraic equations. Consider the order $k$ differential equation ({\small DE}) \be\label{PDE}F(t,\zf(t),d_t\zf,\ldots, d_t^k\zf)=F(t,\zf,\zf',\ldots,\zf^{(k)})|_{j^k\zf}=0\;,\ee where $(t,\zf,\zf',\ldots,\zf^{(k)})$ are coordinates of the $k$-th jet space $J^kE$ and where $j^k\zf$ is the $k$-jet of the section $\zf(t)$. Note that the algebraic equation \be\label{AlgE}F(t,\zf,\zf',\ldots,\zf^{(k)})=0\ee defines a `surface' $\cE^k\subset J^kE$, and that a solution of the considered {\small DE} is nothing but a section $\zf(t)$ whose $k$-jet is located on $\cE^k$.\medskip

A second fundamental feature is that one prefers replacing the original system of {\small PDE}-s by an enlarged system, its infinite prolongation, which also takes into account the consequences of the original one. More precisely, if $\zf(t)$ satisfies the original {\small PDE}, we have also $$d^\ell_t(F(t,\zf(t),d_t\zf,\ldots,d_t^k\zf))=(\p_t+\zf'\p_\zf+\zf''\p_{\zf'}+\ldots)^\ell F(t,\zf,\zf',\ldots,\zf^{(k)})|_{j^\infty\zf}=:$$ \be\label{ProlPDE}D_t^\ell F(t,\zf,\zf',\ldots,\zf^{(k)})|_{j^\infty\zf}=0,\;\forall \ell\in \N\;.\ee Let us stress that the `total derivative' $D_t$ or horizontal lift $D_t$ of $d_t$ is actually an infinite sum. The two systems of {\small PDE}-s, (\ref{PDE}) and (\ref{ProlPDE}), have clearly the same solutions, so we may focus just as well on (\ref{ProlPDE}). The corresponding algebraic system \be\label{ProlAlgE}D_t^\ell F(t,\zf,\zf',\ldots,\zf^{(k)})=0,\;\forall\ell\in\N\;\ee defines a `surface' $\cE^\infty$ in the infinite jet bundle $\zp_\infty:J^\infty E\to X$. A solution of the original system (\ref{PDE}) is now a section $\zf\in\zG(X,E)$ such that $(j^\infty\zf)(X)\subset{\cal E}^\infty$. The `surface' $\cE^\infty$ is often referred to as the `stationary surface' or the `shell'.\medskip

The just described passage from prolonged {\small PDE}-s to prolonged algebraic equations involves the lift of differential operators $d_t^\ell$ acting on $\cO(X)=\zG(X,X\times\R)$ (resp., sending -- more generally -- sections $\zG(X,G)$ of some vector bundle to sections $\zG(X,K)$), to horizontal differential operators $D_t^\ell$ acting on $\cO(J^\infty E)$ (resp., acting from $\zG(J^\infty E,\zp_\infty^*G)$ to $\zG(J^\infty E,\zp_\infty^*K)$). As seen from Equation (\ref{ProlPDE}), this lift is defined by $$(D_t^\ell F)\circ{j^\infty\zf}=d_t^\ell(F\circ j^\infty \zf)\;$$ (note that composites of the type $F\circ j^\infty\zf$, where $F$ is a section of the pullback bundle $\zp_\infty^* G$, are sections of $G$). The interesting observation is that the jet bundle formalism naturally leads to a systematic base change $X\rightsquigarrow J^\infty E$. The remark is fundamental in the sense that both the classical Koszul-Tate resolution (i.e., the Tate extension of the Koszul resolution of a regular surface) and Verbovetsky's Koszul-Tate resolution (i.e., the resolution induced by the compatibility complex of the linearization of the equation), use the jet formalism to resolve on-shell functions $\cO(\cE^\infty)$, and {thus contain the} base change $\bullet\to X$ $\;\rightsquigarrow\;$ $\bullet\to J^\infty E$. This means, dually, that we pass from $\tt DG\cD A$, i.e., from the coslice category $\cO(X)\downarrow \tt DG\cD A$ to the coslice category $\cO(J^\infty E)\downarrow \tt DG\cD A$.

\subsection{{Revisiting classical} Koszul-Tate resolution}

We first recall the local construction of the {\bf Koszul resolution} of the function algebra $\cO(\zS)$ of a regular surface $\zS\subset\R^n$. Such a surface $\zS$, say of codimension $r$, can locally always be described -- in appropriate coordinates -- by the equations \be\label{E}\zS:x^a=0,\;\forall a\in\{1,\ldots,r\}\;.\ee The Koszul resolution of $\cO(\zS)$ is then the chain complex made of {the free Grassmann algebra, i.e., the free graded commutative algebra} $$\op{K}=\cO(\R^n)\otimes \cS[\zf^{a*}]$$ on $r$ odd generators $\zf^{a*}$ -- associated to the equations (\ref{E}) -- and of the Koszul differential \be\label{KDiff}\zd_{\op{K}}=x^a\p_{\zf^{a*}}\;.\ee Of course, the claim that this complex is a resolution of $\cO(\zS)$ means that the homology of $(\op{K},\zd_{\op{K}})$ is given by \be\label{Homology}H_0(\op{K})=\cO(\zS)\quad\text{and}\quad H_k(\op{K})=0,\;\forall k>0\;.\ee\medskip

The {\bf Koszul-Tate resolution} of the algebra $\cO(\cE^\infty)$ of on-shell functions is a generalization of the preceding Koszul resolution. In gauge field theory (our main target), $\cE^\infty$ is the stationary surface given by a system \be\label{Eq}\cE^\infty: D_x^\za F_i=0,\;\forall \za,i\;\ee of prolonged algebraized (see (\ref{ProlAlgE})) Euler-Lagrange equations that correspond to some action functional (here $x\in\R^p$ and $\za\in\N^p$). However, there is a difference between the situations (\ref{E}) and (\ref{Eq}): in the latter, there exist gauge symmetries that implement Noether identities and their extensions -- i.e., extensions \be\label{NI}D_x^\zb\; G_{j\za}^i\,D_x^\za F_i=0,\;\forall \zb,j\;\ee of $\cO(J^\infty E)$-linear relations $G_{j\za}^i\,D_x^\za F_i=0$ between the equations $D_x^\za F_i=0$ of $\cE^\infty$, which do not have any counterpart in the former. It turns out that, to kill the homology (see (\ref{Homology})), we must introduce additional generators that take into account these relations. More precisely, we do not only associate degree 1 generators $\zf^{\za*}_i$ to the equations (\ref{Eq}), but assign further degree 2 generators $C^{\zb*}_j$ to the relations (\ref{NI}). The Koszul-Tate resolution of $\cO(\cE^\infty)$ is then (under appropriate irreducibility and regularity conditions) the chain complex, whose chains are the elements of the free Grassmann algebra \be\label{KT}\op{KT}=\cO(J^{\infty}E)\otimes \cS[\zf^{\za*}_i,C^{\zb*}_j]\;,\ee and whose differential is defined in analogy with (\ref{KDiff}) by \be\label{KTDiff}\zd_{\op{KT}}=D^\za_xF_i\;\p_{\zf^{\za*}_i}+D_x^\zb\; G^i_{j\za}\,D_x^\za \zf^{*}_i\;\p_{C^{\zb*}_j}\;,\ee where we substituted $\zf^{*}_i$ to $F_i$ (and where total derivatives have to be interpreted in the extended sense that puts the `antifields' $\zf^{*}_i$ and $C^{*}_j$ on an equal footing with the `fields' $\zf^i$ (fiber coordinates of $E$), i.e., where we set $$D_{x^k}=\p_{x^k}+\zf^i_{k\za}\p_{\zf^i_\za}+\zf^{k\za *}_i\p_{\zf^{\za *}_i}+C^{k\zb *}_j\p_{C^{\zb *}_j}\,)\;.$$ The homology of this Koszul-Tate chain complex is actually concentrated in degree 0, where it coincides with $\cO(\cE^\infty)$ (compare with (\ref{Homology})) {\cite{HT}}.

\subsection{$\cD$-algebraic version of the Koszul-Tate resolution}

In this subsection, we briefly report on the $\cD$-algebraic approach to `Koszul-Tate' (see \cite{PP} for additional details).

\begin{prop} The functor $$\op{For}:\tt \cD A\to \cO A$$ has a left adjoint $${\cal J}^{\infty}:\tt \cO A\to \cD A\;,$$ i.e., for $B\in\tt \cO A$ and $A\in\tt \cD A$, we have \be\label{JetAlg}\h_{\tt\cD A}({\cal J}^{\infty}(B),A)\simeq \h_{\tt \cO A}(B,\op{For}(A))\;,\ee functorially in $A,B$.\end{prop}

Let now $\zp:E\to X$ be a smooth map of smooth affine algebraic varieties (or a smooth vector bundle). The function algebra $B=\cO(E)$ (in the vector bundle case, we only consider those smooth functions on $E$ that are polynomial along the fibers, i.e., $\cO(E):=\zG(\cS E^*)$) is canonically an $\cO$-algebra, so that the jet algebra ${\cal J}^{\infty}(\cO(E))$ is a $\cD$-algebra. The latter can be thought of as the $\cD$-algebraic counterpart of $\cO(J^\infty E)$. Just as we considered above a scalar {\small PDE} with unknown in $\zG(E)$ as a function $F\in\cO(J^\infty E)$ (see (\ref{AlgE})), an element $P\in{\cal J}^{\infty}(\cO(E))$ can be viewed as a polynomial {\small PDE} acting on sections of $\zp:E\to X$. Finally, the $\cD$-algebraic version of on-shell functions $\cO(\cE^\infty)=\cO(J^\infty E)/(F)$ is the quotient ${\cal R}(E,P):=\cJ^\infty(\cO(E))/(P)$ of the jet $\cD$-algebra by the $\cD$-ideal $(P)$.\medskip

A first candidate for a Koszul-Tate resolution of ${\cal R}:={\cal R}(E,P)\in\tt \cD A$ is of course the cofibrant replacement of ${\cal R}$ in $\tt DG\cD A$ given by the functorial `{\small Cof -- TrivFib}' factorization of Theorem \ref{P:c-tf_tc-f}, when applied to the canonical $\tt DG\cD A$-morphism $\cO\to {\cal R}$. Indeed, this decomposition implements a functorial cofibrant replacement functor $Q$ (see Theorem \ref{FCRF} below) with value $Q({\cal R})=\cS V$ described in Theorem \ref{P:c-tf_tc-f}: $$\cO\rightarrowtail \cS V\stackrel{\sim}{\twoheadrightarrow}{\cal R}\;.$$ Since ${\cal R}$ is concentrated in degree 0 and has 0 differential, it is clear that $H_k(\cS V)$ vanishes, except in degree 0 where it coincides with ${\cal R}$.\medskip

As already mentioned, we propose a general and precise definition of a Koszul-Tate resolution in \cite{PP}. Although such a definition does not seem to exist in the literature, {the classical Koszul-Tate resolution of the quotient of a commutative ring $k$ by an ideal $I$ is a $k$-algebra that resolves $k/I$}.\medskip

The natural idea -- to get a ${\cal J}^\infty(\cO(E))$-algebra -- is to replace $\cS V$ by ${\cal J}^\infty(\cO(E))\0 \cS V$, and, more precisely, to consider the `{\small Cof -- TrivFib}' decomposition $${\cal J}^\infty(\cO(E))\rightarrowtail {\cal J}^\infty(\cO(E))\0 \cS V\stackrel{\sim}{\twoheadrightarrow}{\cal J}^\infty(\cO(E))/(P)\;.$$ The {\small DG$\cD$A} \be\label{KTD} {\cal J}^\infty(\cO(E))\0 \cS V\ee {\it is} a {\bf $\cJ^\infty(\cO(E))$-algebra} that {\bf resolves} ${\cal R}=\cJ^\infty(\cO(E))/(P)$, but it is of course {\it not} a cofibrant replacement, since the left algebra is not the initial object $\cO$ in $\tt DG\cD A$ (further, the considered factorization does not canonically induce a cofibrant replacement in $\tt DG\cD A$, since it can be shown that the morphism $\cO\to {\cal J}^\infty(\cO(E))$ is not a cofibration). However, as emphasized above, the Koszul-Tate problem requires a passage from $\tt DG\cD A$ to ${\cal J}^\infty(\cO(E))\downarrow \tt DG\cD A$. It is easily checked that, in the latter undercategory, ${\cal J}^\infty(\cO(E))\0 \cS V$ is a {\bf cofibrant replacement} of ${\cal J}^\infty(\cO(E))/(P)$. To further illuminate the $\cD$-algebraic approach to Koszul-Tate, let us mention why the complex (\ref{KT}) is of the same type as (\ref{KTD}). Just as the variables $\zf^{(k)}$ (see (\ref{PDE})) are algebraizations of the derivatives $d_t^k\zf$ of a section $\zf$ of a vector bundle $E\to X$ (fields), the generators $\zf^{\za*}_i$ and $C^{\zb*}_j$ (see (\ref{Eq}) and (\ref{NI})) symbolize the total derivatives $D_x^\za\zf^*_i$ and $D_x^\zb C^*_j$ of sections $\zf^*$ and $C^*$ of some vector bundles $\zp_\infty^*F_1\to J^\infty E$ and $\zp_\infty^*F_2\to J^\infty E$ (antifields). Hence, the $\zf^{\za*}_i$ and $C^{\zb*}_j$ can be thought of as the horizontal jet bundle coordinates of $\zp_\infty^*F_1$ and $\zp_\infty^*F_2\,$. These coordinates may of course be denoted by other symbols, e.g., by $\p_x^\za\cdot\zf_i^*$ and $\p_x^\zb\cdot C_j^*$, provided {we define the $\cD$-action $\cdot$ as} the action $D_x^\za\zf^*_i$ and $D_x^\zb C^*_j$ by the corresponding horizontal lift, so that we get appropriate interpretations when the $\zf^*_i$-s and the $C^*_j$-s are the components of true sections. This convention allows us to write {$$\op{KT}=J\0\cS[\p_x^\za\cdot\zf_i^*,\p_x^\zb\cdot C_j^*]=J\0_\cO\cS_\cO(\oplus_i\,\cD\cdot\zf_i^*\;\bigoplus\; \oplus_j\,\cD\cdot C_j^*)\;,$$} where $J={\cal J}^\infty(\cO(E))\,,$ so that the space (\ref{KT}) is really of the type (\ref{KTD}). Let us emphasize that (\ref{KT}) and (\ref{KTD}), although of the same type, are of course not equal (for instance, the classical Koszul-Tate resolution is far from being functorial). For further details, see \cite{PP}.

\section{Appendices}

The following appendices do not contain new results but might have a pedagogical value. Various (also online) sources were used. Notation is the same as in the main part of the text.

\subsection{Appendix 1 -- Quasi-coherent sheaves of modules}\label{FinCondShMod}

A {\bf quasi-coherent $\cR$-module} is an object $\cP\in{\tt Mod}(\cR)$ that is locally presented, i.e., for any $x\in X$, there is a neighborhood $U\ni x$, such that there is an exact sequence of sheaves
\be\label{QuaCoh}\cR^{K_U}|_{U}\to \cR^{J_U}|_{U}\to \cP|_{U}\to 0\;,\ee where $\cR^{K_U}$ and $\cR^{J_U}$ are (not necessarily finite) direct sums. Let us recall that an infinite direct sum of sheaves need not be a sheaf, so that a sheafification is required. The category ${\tt qcMod}(\cR)$ of quasi-coherent $\cR$-modules is not abelian in general, but is abelian in the context of Algebraic Geometry, i.e., if $\cR$ is the function sheaf of a scheme.

\subsection{Appendix 2 -- $\cD$-modules}\label{D-modules}

We already indicated that $\cD$-modules are fundamental in algebraic analysis: they allow us to apply methods of homological algebra and sheaf theory to the study of systems of {\small PDE}-s \cite{KS}.\medskip

We first explain the key idea of Proposition \ref{DModFlatConnSh} considering -- to simplify -- {global sections} instead of sheaves.\medskip

We denote by $\cD$ the ring of differential operators acting on functions of a suitable base space $X$, e.g., a finite-dimensional smooth manifold \cite{Cos}. A {$\cD$-module} $M\in{\tt Mod}(\cD)$ (resp., $M\in{\tt Mod}(\cD^{\op{op}})$) is a left (resp., right) module over the noncommutative ring $\cD$. Since {\it $\cD$ is generated by smooth functions $f\in\cO$ and smooth vector fields $\theta\in\Theta$, modulo the obvious commutation relations between these types of generators, a $\cD$-action on an $\cO$-module $M\in{\tt Mod}(\cO)$ is completely defined if it is given for vector fields and satisfies the natural compatibility conditions}.
More precisely, let $$\cdot\,:\cO\times M\ni(f,m)\mapsto f\cdot m\in M$$ be the $\cO$-action, and let \be\label{FC1}\nabla:\Theta\times M\ni (\theta,m)\mapsto \nabla_\theta m\in M\ee be an $\R$-bilinear `$\Theta$-action'. For $f\in\cO$ and $\theta,\theta'\in\Theta$, we then necessarily extend $\nabla$ by defining the action $\nabla_{\theta\theta'}$ (resp., $\nabla_{\theta f}$) of the differential operator $\theta\theta'=\theta\circ\theta'$ (resp., $\theta f=\theta\circ f$) by $$\nabla_{\theta\theta'}:=\nabla_\theta\nabla_{\theta'}$$ (resp., $$\nabla_{\theta f}:=\nabla_\theta(f\,\cdot\,-))\;.$$ Since $f\zy=f\circ\zy$, we get the compatibility condition \be\label{FC2}\nabla_{f\theta}=f\cdot\nabla_\theta\;,\ee and, as $\theta f=f\theta+\theta(f)$ (resp., $\theta \theta'=\theta'\theta+[\theta,\theta']$) -- where $\theta(f)$ (resp., $[\theta,\theta']$) denotes the Lie derivative $L_\theta f$ of $f$ with respect to $\theta$ (resp., the Lie bracket of the vector fields $\theta,\theta'$), we also find the compatibility relations \be\label{FC3}\nabla_{\theta}(f\,\cdot\,-)=f\cdot \nabla_{\theta}+\theta(f)\,\cdot\,-\,\ee (resp., \be\label{FC4}\nabla_\theta\nabla_{\theta'}=\nabla_{\theta'}\nabla_\theta+\nabla_{[\theta,\theta']})\;.\ee Hence, if the compatibility conditions (\ref{FC2}) -- (\ref{FC4}) hold, we defined the unique structure of left $\cD$-module on $M$ that extends the `action of $\Theta$'. In view of Equations (\ref{FC1}) -- (\ref{FC4}), {\it a $\cD$-module structure on $M\in{\tt Mod}(\cO)$ is the same as a flat connection on $M$}. \medskip

When resuming now our explanations given in Subsection \ref{D-ModulesAlgebras}, we understand that a morphism $\nabla$ of sheaves of $\K$-vector spaces satisfying the conditions (1) -- (3) is exactly a family of $\cD_X(U)$-modules $\cM_X(U)$, $U\in{\tt Open}_X$, such that the $\cD_X(U)$-actions are compatible with restrictions, i.e., is exactly a $\cD_X$-module structure on the considered sheaf $\cM_X$ of $\cO_X$-modules.\medskip

Typical examples of $\mathcal{D}$-modules are:
\begin{itemize}
\item $\cO\in{\tt Mod}(\cD)$ with action $\nabla_\theta=L_\theta$,
\item the top differential forms $\zW^{\op{top}}\in{\tt Mod}(\cD^{\op{op}})$ with action $\nabla_\theta=-L_\theta$, and
\item $\cD\in{\tt Mod}(\cD)\cap{\tt Mod}(\cD^{\op{op}})$ with action given by left and right compositions.
\end{itemize}

\subsection{Appendix 3 -- Sheaves versus global sections}\label{ShVsGlobSec}

In Classical Differential Geometry, the fundamental spaces (resp., operators), e.g., vector fields, differential forms... (resp., the Lie derivative, the de Rham differential...) are sheaves (resp., sheaf morphisms). Despite this sheaf-theoretic nature, most textbooks present Differential Geometry in terms of global sections and morphisms between them. Since these sections are sections of vector bundles (resp., these global morphisms are local operators), restriction and gluing is canonical (resp., the existence of smooth bump functions allows us to localize the global morphisms in such a way that they commute with restrictions; e.g., for the de Rham differential, we have $$(d|_U\zw_U)|_V=\lp d(\za_V\zw_U)\rp|_V\quad{\text{and}}\quad d|_U\zw|_U=(d\zw)|_U\;,$$ where $\za_V$ is a bump function with constant value 1 in $V\subset U$ and support in $U$). Such global viewpoints are not possible in the real-analytic and holomorphic settings, since no interesting analytic bump {functions exist}.\medskip

{There are a number} of well-known results on the equivalence of categories of sheaves and the corresponding categories of global sections, essentially when the topological space underlying the considered sheaves is an affine scheme or variety. In the present paper, we use the fact that, for an affine scheme $(X,\cO_X)$, there is an equivalence \cite{Har} \be\label{ShVsSectqcMod}\zG(X,\bullet): {\tt qcMod}(\cO_X)\rightleftarrows {\tt Mod}(\cO_X(X)):\widetilde{\bullet}\ee between the category of quasi-coherent $\cO_X$-modules and the category of $\cO_X(X)$-modules. The functor $\widetilde{\bullet}$ is isomorphic to the functor $\cO_X\0_{\cO_X(X)}\bullet\,$.

\subsection{Appendix 4 -- Model categories}\label{ModCat}

Quite a few non-equivalent definitions of model categories and cofibrantly generated model categories can be found in the literature. In this paper, we use the definitions of \cite{Hov} and of \cite{GS}.\medskip

In the definition of {\bf model categories}, both texts \cite{Hov} and \cite{GS} assume the existence of all small limits and colimits in the underlying category -- in contrast with Quillen's original definition, which asks only for the existence of finite limits and colimits. However, the two references use different `cofibration - trivial fibration' and `trivial cofibration - fibration' factorization axioms {\small MC5}. Indeed, in \cite{GS}, the authors use Quillen's original axiom, which merely requires the existence of the two factorizations, whereas in \cite{Hov}, the author requires the factorizations to be functorial, and even includes the choice of a pair of such functorial factorizations in the axioms of the model structure. However, this difference does not play any role in the present paper, since we are dealing with cofibrantly generated model categories, so that a choice of functorial factorizations is always possible via the small object argument.\medskip

For the definitions of cofibrantly generated model structures, some preparation is needed.\medskip

An {\bf ordinal $\zl$ is filtered with respect to a cardinal $\zk$}, if $\zl$ is a limit ordinal such that the supremum of a subset of $\zl$ of cardinality at most $\zk$ is smaller than $\zl$. {\it This condition is actually a largeness condition for $\zl$ with respect to $\zk$}: if $\zl$ is $\zk$-filtered for $\zk>\zk'$, then $\zl$ is also $\zk'$-filtered. For a finite cardinal $\zk$, a $\zk$-filtered ordinal is just a limit ordinal.\medskip

Smallness of an object $A$ in a category $\tt C$ (assumed to have all small colimits) is defined with respect to a class of morphisms $W$ in $\tt C$ and a cardinal $\zk$ (that can depend on $A$) \cite{Hov}. The point is that the covariant Hom-functor $${\tt C}(A,\bullet):=\op{Hom}_{\tt C}(A,\bullet)$$ commutes with limits, but usually not with colimits. However, {\it if the considered sequence is sufficiently large with respect to $A$, then commutation may be proven}. More precisely, for $A\in\tt C$, we consider the colimits of all the $\zl$-sequences (with arrows in $W$) for all $\zk$-filtered ordinals $\zl$ (usually for $\zk=\zk(A)$), and try to prove that the covariant Hom-functor ${\tt C}(A,\bullet)$ commutes with these colimits. In this case, we say that {\bf $A\in\tt C$ is small with respect to $\zk$ and $W$}. Of course, {\it if $\zk<\zk'$, then $\zk$-smallness implies $\zk'$-smallness}.\medskip

In \cite{GS}, `small' (with respect to $W$) means `sequentially small': the covariant Hom-functor commutes with the colimits of the $\zw$-sequences. This concept matches the notion `$n$-small', i.e., small relative to a finite cardinal $n\in\N$: the covariant Hom-functor commutes with the colimits of the $\zl$-sequences for all limit ordinals $\zl$. In \cite{Hov}, `small' (relative to $W$) means $\zk$-small for some $\zk$: the covariant Hom-functor commutes with the colimits of all the $\zl$-sequences for all the $\zk$-filtered ordinals $\zl$. It is clear that $n$-small implies $\zk$-small, for any $\zk>n$.\medskip

More precisely, a $\zl$-sequence in $\tt C$ is a colimit respecting functor $X:\zl\to \tt C$. Usually this diagram is denoted by $X_0\to X_1\to\ldots\to X_\zb\to \ldots$ It is natural to refer to the map $$X_0\to \op{colim}_{\zb<\zl}X_\zb$$ as the composite of the $\zl$-sequence $X$. If $W$ is a class of morphisms in $\tt C$ and every map $X_\zb\to X_{\zb+1}$, $\zb+1<\zl$, is in $W$, we refer to the composite $X_0\to \op{colim}_{\zb<\zl}X_\zb$ as a transfinite composition of maps in $W$. Let us also recall that, if we have a commutative square in $\tt C$, the right down arrow is said to be the pushout of the left down arrow. We now denote by {\bf $W$-cell} the class of {\bf transfinite compositions of pushouts of arrows in $W$}. It turns out that $W$-cell is a subclass of the class $\op{LLP}(\op{RLP}(W))$ (where notation is self-explaining).\medskip

We are now prepared to give the finite and the transfinite definitions of a cofibrantly generated model category.\smallskip

A model category is {\bf cofibrantly generated} \cite{GS}, if there exist {\it sets} of morphisms $I$ and $J$, which generate the cofibrations and the trivial cofibrations, respectively, i.e., more precisely, if there are sets $I$ and $J$ such that

\begin{enumerate}

\item the source of every morphism in $I$ is sequentially small with respect to the class $\op{Cof}$, and $\op{TrivFib}=\op{RLP}(I)\,$,

\item the source of every morphism in $J$ is sequentially small with respect to the class $\op{TrivCof}$, and $\op{Fib}=\op{RLP}(J)\,$.

\end{enumerate}
It then follows that $I$ and $J$ are actually the generating cofibrations and the generating trivial cofibrations: $$\op{Cof}=\op{LLP}(\op{RLP}(I))\quad\text{and}\quad\op{TrivCof}=\op{LLP}(\op{RLP}(J))\;.$$

Alternatively, a model category is {\bf cofibrantly generated} \cite{Hov}, if there exist sets $I$ and $J$ of maps such that

\begin{enumerate}
\item the domains of the maps in $I$ are small ($\zk$-small for some fixed $\zk$) relative to $I$-cell, and $\op{TrivFib}=\op{RLP}(I)\,$,

\item the domains of the maps in $J$ are small ($\zk$-small for some fixed $\zk$) relative to $J$-cell, and $\op{Fib}=\op{RLP}(J)\,$.
\end{enumerate}
It is clear that the finite definition \cite{GS} is stronger than the transfinite one \cite{Hov}. First, $n$-smallness implies $\zk$-smallness, and, second, smallness with respect to $\op{Cof}$ (resp., $\op{TrivCof}$) implies smallness with respect to $I$-cell (resp., $J$-cell).\medskip

The model structures we study in the present paper are all {\it finitely} generated. A {\bf finitely generated model structure} \cite{Hov} is a cofibrantly generated model structure, such that $I$ and $J$ can be chosen so that their sources and targets are $n$-small, $n\in\N$, relative to $\op{Cof}$. This implies in particular that our model structures are cofibrantly generated in the sense of \cite{GS}.\medskip

For more information on model categories, we refer the reader to \cite{GS}, \cite{Hir}, \cite{Hov}, and \cite{Quill}. The background material on category theory can be found in \cite{Bor1}, \cite{Bor2}, and \cite{Mac}.

\subsection{Appendix 5 -- Invariants versus coinvariants}\label{InvCoinv}

If $G$ is a (multiplicative) group and $k$ a commutative unital ring, we denote by $k[G]$ the group $k$-algebra of $G$ (the free $k$-module made of all formal finite linear combinations $\sum_{g\in G} r(g)\, g$ with coefficients in $k$, endowed with the unital ring multiplication that extends the group multiplication by linearity).\medskip

In the following, we use notation of Subsection \ref{Adjunction}. Observe that $\0^n_\cO M_\bullet$ is a module over the group $\cO$-algebra $\cO[\mathbb{S}_n]$, where $\mathbb{S}_n$ denotes the $n$-th symmetric group. There is an $\cO$-module isomorphism \be\label{Iso1}{\cal S}_\cO^n M_\bullet=\0_\cO^n M_\bullet/{\cal I}\cap \0_\cO^n M_\bullet\simeq (\0_\cO^n M_\bullet)_{\mathbb{S}_n}:=\0_\cO^n M_\bullet/\langle T-\zs\cdot T\ra\;,\ee where $(\0_\cO^n M_\bullet)_{\mathbb{S}_n}$ is the $\cO$-module of $\mathbb{S}_n$-coinvariants and where the denominator is the $\cO$-submodule generated by the elements of the type $T-\zs\cdot T$, $T\in\0_\cO^n M_\bullet$, $\zs\in\mathbb{S}_n$ (a Koszul sign is incorporated in the action of $\zs$). It is known that, since the cardinality of $\mathbb{S}_n$ is invertible in $\cO$, we have also an $\cO$-module isomorphism \be\label{Iso2}(\0_\cO^n M_\bullet)_{\mathbb{S}_n}\simeq(\0_\cO^n M_\bullet)^{\mathbb{S}_n}:=\{T\in \0_\cO^n M_\bullet: \zs\cdot T=T,\forall \zs\in\mathbb{S}_n\}\;\ee between the $\mathbb{S}_n$-coinvariants and the $\mathbb{S}_n$-invariants. The averaging map or graded symmetrization operator \be\label{SymOp}{\frak S}:\0_\cO^n M_\bullet\ni T\mapsto \frac{1}{n!}\sum_{\zs\in\mathbb{S}_n}\zs\cdot T\in(\0_\cO^n M_\bullet)^{\mathbb{S}_n}\;\ee coincides with identity on $(\0_\cO^n M_\bullet)^{\mathbb{S}_n}$, what implies that it is surjective. When viewed as defined on coinvariants $(\0_\cO^n M_\bullet)_{\mathbb{S}_n}\,$, it provides the mentioned isomorphism (\ref{Iso2}). It is straightforwardly checked that the graded symmetric multiplication $\vee$ on $(\0_\cO^\ast M_\bullet)^{\mathbb{S}_\ast}$, defined by \be\label{Mult2}{\frak S}(S)\vee {\frak S}(T)={\frak S}({\frak S}(S)\0 {\frak S}(T))\;,\ee endows $(\0_\cO^\ast M_\bullet)^{\mathbb{S}_\ast}$ with a {\small DG} $\cD$-algebra structure, and that the $\cO$-module isomorphism \be\label{Alg2}{\cal S}_\cO^\ast M_\bullet\simeq \{T\in \0_\cO^\ast M_\bullet: \zs\cdot T=T,\forall \zs\in\mathbb{S}_\ast\}\ee is in fact a $\tt DG\cD A$-isomorphism.

\subsection{Appendix 6 -- Proof of Theorem \ref{P:c-tf_tc-f}}\label{CofTrivFibProof}

The proof of functoriality of the decompositions will be given in Appendix \ref{FunctReplFunct}. Thus, only Part (2) requires immediate explanations. We use again the above-introduced notation $R_k=A\0\cS V_k$; we also set $R=A\0\cS V$. The multiplication in $R_k$ (resp., in $R$) will be denoted by $\diamond_k$ (resp., $\diamond$). \medskip

To show that $j$ is a {\small RS$\cD\!$A}, we have to check that $A$ is a differential graded $\cD$-subalgebra of $R$, that the basis of $V$ is indexed by a well-ordered set and that $d_2$ is lowering.

The main idea to keep in mind is that $R=\bigcup_kR_k\,$ -- so that any element of $R$ belongs to some $R_k$ in the increasing sequence $R_0\subset R_1\subset \ldots\,$ -- and that the {\small DG$\cD$A} structure on $R$ is defined in the standard manner. For instance, the product of ${\frak a}\0 X,{\frak b}\0 Y\in R\cap R_k$ is defined by $$({\frak a}\0 X)\diamond ({\frak b}\0 Y)=({\frak a}\0 X)\diamond_k ({\frak b}\0 Y)=(-1)^{\tilde X\tilde{\frak b}}({\frak a}\ast{\frak b})\0 (X\odot Y)\;,$$ where `tilde' (resp., $\ast$) denotes as usual the degree (resp., the multiplication in $A$). It follows that $\diamond$ restricts on $A$ to $\ast\,$. Similarly, $d_2|_A=\zd_0|_A=d_A$, in view of (\ref{DefD2}) and (\ref{NewgTot}). Finally, we see that $A$ satisfies actually the mentioned subalgebra condition.

We now order the basis of $V$. First, we well-order, for any fixed generator degree $m\in\N$ (see (\ref{GenDeg})), the sets \be\label{OrdSetsGen}\{s^{-1}\mbi_{b_{m+1}}\},\, \{\mbi_{b_m}\},\, \{\mbi_{\zb_m}\},\, \{\mbi^1_{\zs_{m-1},\fb_{m}}\},\, \{\mbi^2_{\zs_{m-1},\fb_{m}}\},\, \ldots\ee of degree $m$ generators of a given type (for $m=0$, only the sets $\{s^{-1}\mbi_{b_1}\}$ and $\{\mbi_{\zb_0}\}$ are non-empty). We totally order the set of all degree $m$ generators by totally ordering its partition (\ref{OrdSetsGen}): $$\{s^{-1}\mbi_{b_{m+1}}\}<\{\mbi_{b_m}\}<\{\mbi_{\zb_m}\}<\{\mbi^1_{\zs_{m-1},\fb_{m}}\}<\{\mbi^2_{\zs_{m-1},\fb_{m}}\}<\ldots\;$$ A total order on the set of all generators (of all degrees) is now obtained by declaring that any generator of degree $m$ is smaller than any generator of degree $m+1$. This total order is a well-ordering, since no infinite descending sequence exists in the set of all generators.

Finally, the differential $d_2$ sends the first and third types of generators (see (\ref{OrdSetsGen})) to 0 and it maps the second type to the first. Hence, so far $d_2$ is lowering. Further, we have $$d_2(\mbi^k_{\zs_{m-1},\fb_{m}})=\zs_{m-1}\in (R_{k-1})_{m-1}\;,$$ where $m-1$ refers to the term of degree $m-1$ in $R_{k-1}$. Since this term is generated by the generators $$\{s^{-1}\mbi_{b_{\ell+1}}\}, \{\mbi_{b_\ell}\},\{\mbi_{\zb_\ell}\},\{\mbi^1_{\zs_{\ell-1},\fb_{\ell}}\},\ldots, \{\mbi^{k-1}_{\zs_{\ell-1},\fb_{\ell}}\}\,,$$ where $\ell < m$, the differential $d_2$ is definitely lowering.\medskip

It remains to verify that the described construction yields a morphism $q:A\0\cS V\to B$ that is actually a trivial fibration.

Since fibrations are exactly the morphisms that are surjective in all positive degrees, and since $q|R_U=q_0|R_U=p$ is degree-wise surjective, it is clear that $q$ is a fibration. As for triviality, let $[\zb_n]\in H(B,d_B)$, $n\ge 0\,$. Since $\mbi_{\zb_n}\in \ker \zd_0\subset \ker d_2$, the homology class $[\mbi_{\zb_n}]\in H(R,d_2)$ makes sense; moreover, $$H(q)[\mbi_{\zb_n}]=[q\mbi_{\zb_n}]=[q_0\mbi_{\zb_n}]=[\zb_n]\;,$$ so that $H(q)$ is surjective. {Finally}, let $[\zs_n]\in H(R,d_2)$ and assume that $H(q)[\zs_n]=0$, i.e., that $q\zs_n\in\im d_B$. Since there is a lowest $k\in\N$ such that $\zs_n\in R_k$, we have $\zs_n\in\ker\zd_k$ and $q_k\zs_n=d_B\fb_{n+1}$, for some $\fb_{n+1}\in B_{n+1}$. Hence, a pair $(\zs_n,\fb_{n+1})\in{\frak B}_k$ and a generator $\mbi^{k+1}_{\zs_n,\fb_{n+1}}\in R_{k+1}\subset R$. Since $$\zs_n=\zd_{k+1}\mbi^{k+1}_{\zs_n,\fb_{n+1}}=d_2\mbi^{k+1}_{\zs_n,\fb_{n+1}}\;,$$ we obtain that $[\zs_n]=0$ and that $H(q)$ is injective.

\subsection{Appendix 7 -- Explicit functorial cofibrant replacement functor}\label{FunctReplFunct}

(1) We proved in Subsection \ref{Condition2} that the factorization $(i,p)=(i(\zf),p(\zf))$ of the $\tt DG\cD A$-morphisms $\zf$, described in Theorem \ref{P:c-tf_tc-f}, is functorial:

\begin{prop} In $\tt DG\cD A$, the functorial fibrant replacement functor $R$, which is induced by the functorial `{\small TrivCof -- Fib}' factorization $(i,p)$ of Theorem \ref{P:c-tf_tc-f}, is the identity functor $R=\id\,$. \end{prop}

(2) To finish the proof of Theorem \ref{P:c-tf_tc-f}, we still have to show that the factorization $(j,q)$ is functorial, i.e., that for any commutative $\tt DG\cD A$-square \be\label{InitSq1}
\xymatrix{
A\ar[d]^{u} \ar[r]^{\zf}&B\ar[d]^{v\;\;\;,}\\
A'\ar[r]^{\zf'}&B'\\
}
\ee
there is a commutative $\tt DG\cD A$-diagram
\be\label{CompMor2}
\xymatrix{A\;\; \ar[d]_{u} \ar_{j:=j(\zf)}  @{>->} [r] & A\0\cS V \ar[d]^{\zw}\;\;\ar @{->>} [r]^{\sim}_{q:=q(\zf)} & B \ar[d]^{v\;\;\;.}\\
A'\;\; \ar @{>->} [r]_{j':=j(\zf')} & A'\0\cS V'\;\; \ar @{->>} [r]^{\sim}_{q':=q(\zf')} & B'\\
}
\ee\smallskip

Let us stress that the following proof fails, if we use the non-functorial factorization mentioned in Remark \ref{NonFuncFact} (the critical spots are marked by $\triangleleft\,$).

Just as we constructed in Section \ref{Factorizations}, the {\small RS$\cD$A} $R=A\0\cS V$ (resp., $R'=A'\0\cS V'$) as the colimit of a sequence $R_k=A\0\cS V_k$ (resp., $R'_k=A'\0\cS V'_k$), we will build $\zw\in{\tt DG\cD A}(R,R')$ as the colimit of a sequence \be\label{w_k}\zw_k\in{\tt DG\cD A}(R_k,R'_k)\;.\ee Recall moreover that $q$ is the colimit of a sequence $q_k\in{\tt DG\cD A}(R_k,B)$, and that $j$ is nothing but $j_k\in{\tt DG\cD A}(A,R_k)$ viewed as valued in the supalgebra $R$ -- and similarly for $q',q'_k,j',j'_k$. Since we look for a morphism $\zw$ that makes the left and right squares of the diagram (\ref{CompMor2}) commutative, we will construct $\zw_k$ so that \be\label{ComSqu}\zw_k\,j_k=j'_k\,u\;\,\text{and}\;\; v\,q_k=q'_k\,\zw_k\;.\ee

Since the {\small RS$\cD$A} $A\to R_0=A\0\cS V_0$ is split, we define $$\zw_0\in{\tt DG\cD A}(A\0\cS V_0, R'_0)$$ as \be\label{zw_0}\zw_0=j'_0\,u\diamond_0 w_0\;,\ee where we denoted the multiplication in $R'_0\,$ by the same symbol $\diamond_0$ as the multiplication in $R_0$, where $j'_0\,u\in{\tt DG\cD A}(A,R'_0)$, and where $w_0\in{\tt DG\cD A}(\cS V_0, R'_0)$. As the differential $\zd_{V_0}$, see Section \ref{Factorizations}, has been obtained via Lemma \ref{DiffGen}, the morphism $w_0$ can be built as described in Lemma \ref{MorpGen}: we set \be\label{w_0}w_0(s^{-1}\mbi_{b_n})=s^{-1}\mbi_{v(b_n)}\in V'_0\,,\;w_0(\mbi_{b_n})=\mbi_{v(b_n)}\in V'_0\,,\;\, \text{and}\;\; w_0(\mbi_{\zb_n})=\mbi_{v(\zb_n)}\in V'_0\;,\ee and easily check that $w_0\,\zd_{V_0}=\zd'_0\, w_0$ on the generators. The first commutation condition (\ref{ComSqu}) is obviously satisfied. As for the verification of the second condition, let $t={\frak a}\0 x_1\odot\ldots\odot x_\ell\in A\0\cS V_0$ and remember (see (\ref{Morp})) that $q_0=\zf\star q_{V_0}$ and $q'_0=\zf'\star q_{V'_0}\,$, where we denoted again the multiplications in $B$ and $B'$ by the same symbol $\star$. Then $$vq_0(t)=v\zf({\frak a})\star vq_{V_0}(x_1)\star\ldots\star vq_{V_0}(x_\ell)$$ and $$q'_0\zw_0(t)=q'_0j'_0u({\frak a})\star q'_0w_0(x_1)\star\ldots\star q'_0w_0(x_\ell)=\zf'u({\frak a})\star q'_0w_0(x_1)\star\ldots\star q'_0w_0(x_\ell)\;.$$ It thus suffices to show that $v\,q_{V_0}=q'_0\,w_0$ on the generators $s^{-1}\mbi_{b_n}, \mbi_{b_n}, \mbi_{\zb_n}$ of $V_0$, what follows from Equations (\ref{Morpg}) and (\ref{w_0}) ($\triangleleft_1$).

Assume now that the $\zw_\ell$ have been constructed according to the requirements (\ref{w_k}) and (\ref{ComSqu}), for all $\ell\in\{0,\dots,k-1\}$, and build their extension $$\zw_k\in{\tt DG\cD A}(R_k,R'_k)$$ as follows. Since $\zw_{k-1}$, viewed as valued in $R'_k$, is a morphism $\zw_{k-1}\in{\tt DG\cD A}(R_{k-1},R'_k)$ and since the differential $\zd_k$ of $R_k\simeq R_{k-1}\0 \cS G_k$, where $G_k$ is the free $\cD$-module $$G_k=\langle \mbi^k_{\zs_n,\fb_{n+1}}:(\zs_n,\fb_{n+1})\in{\frak B}_{k-1}\ra\;,$$ has been defined by means of Lemma \ref{LemRSA}, the morphism $\zw_k$ is, in view of the same lemma, completely defined by degree $n+1$ values $$\zw_k(\mbi^k_{\zs_n,\fb_{n+1}})\in \zd'^{-1}_k(\zw_{k-1}\zd_k(\mbi^k_{\zs_n,\fb_{n+1}}))\;.$$ As the last condition reads $$\zd'_k\,\zw_k(\mbi^k_{\zs_n,\fb_{n+1}})=\zw_{k-1}(\zs_n)\;,$$ it is natural to set \be\label{w_kDef}\zw_k(\mbi^k_{\zs_n,\fb_{n+1}})=\mbi^k_{\zw_{k-1}(\zs_n),v(\fb_{n+1})}\;,\ee provided we have $$(\zw_{k-1}(\zs_n),v(\fb_{n+1}))\in{\frak B}'_{k-1}\quad (\triangleleft_2)\;.$$ This requirement means that $\zd'_{k-1}\zw_{k-1}(\zs_n)=0$ and that $q'_{k-1}\zw_{k-1}(\zs_n)=d_{B'}\,v(\fb_{n+1})$. To see that both conditions hold, it suffices to remember that $(\zs_n,\fb_{n+1})\in{\frak B}_{k-1}$, that $\zw_{k-1}$ commutes with the differentials, and that it satisfies the second equation (\ref{ComSqu}). Hence the searched morphism $\zw_k\in{\tt DG\cD A}(R_k,R'_k)$, such that $\zw_k|_{R_{k-1}}=\zw_{k-1}$ (where the {\small RHS} is viewed as valued in $R'_k$). To finish the construction of $\zw_k$, we must still verify that $\zw_k$ complies with (\ref{ComSqu}). The first commutation relation is clearly satisfied. For the second, we consider $$r_k=r_{k-1}\0 g_1\odot\ldots\odot g_\ell\in R_{k-1}\0\cS G_k\;$$ and proceed as above: recalling that $\zw_k$ and $q_k$ have been defined via Equation (\ref{DefRSAMorph}) in Lemma \ref{LemRSA}, that $q'_k$ and $v$ are algebra morphisms, and that $\zw_{k-1}$ satisfies (\ref{ComSqu}), we see that it suffices to check that $q'_k\,\zw_k=v\,q_k$ on the generators $\mbi^k_{\zs_n,\fb_{n+1}}$ -- what follows immediately from the definitions ($\triangleleft_3$).

Remember now that $((R,d_2),i_r)$ is the direct limit of the direct system $((R_k,\zd_k),\iota_{sr})$, i.e., that
 \begin{equation} \begin{tikzpicture}
 \matrix (m) [matrix of math nodes, row sep=3em, column sep=3em]
   {  R_0  & \cdots & R_k & \cdots  \\
       & R \\ };
 \path[>->]
 (m-1-1) edge node[right] {\small{$i_0$}} (m-2-2)
 (m-1-3) edge node[right] {\small{$i_k$}} (m-2-2)
 (m-1-3) edge node[above] {\small{$\iota_{k+1,k}$}} (m-1-4)
 (m-1-1) edge node[above] {\small{$\iota_{10}$}}(m-1-2)
 (m-1-2) edge node[above] {\small{$\iota_{k,k-1}$}}(m-1-3);
\end{tikzpicture}\;
\end{equation}
where all arrows are canonical inclusions, and that the same holds for $((R',d'_2),i'_r)$ and $((R'_k,\zd'_k),\iota'_{sr})$. Since the just defined morphisms $\zw_k$ provide morphisms $i'_k\,\zw_k\in{\tt DG\cD A}(R_k,R')$ (such that the required commutations hold -- as $\zw_k|_{R_0}=\zw_0$), it follows from universality that there is a unique morphism $\zw\in{\tt DG\cD A}(R,R')$, such that $\zw\,i_k=i'_k\,\zw_k\,$, i.e., such that \be\label{w}\zw|_{R_k}=\zw_k\;.\ee When using the last result, one easily concludes that $\zw\,j=j'\,u$ and $v\,q=q'\,\zw\,$.\medskip

This completes the proof of Theorem \ref{P:c-tf_tc-f}.\medskip

\begin{rem} The preceding proof of functoriality fails for the factorization of Remark \ref{NonFuncFact}. The latter adds only one new generator $\mbi_{\dot\zb_n}$ for each homology class $\dot\zb_n\simeq[\zb_n]$, and it adds only one new generator $\mbi^k_{\zs_n}$ for each $\zs_n\in {\cal B}_{k-1}\setminus {\cal B}_{k-2}\,$, where $${\cal B}_r=\{\zs_n\in\ker\zd_r:q_r\zs_n\in \im d_B,n\ge 0\}\;.$$ In $(\,$$\triangleleft_1$$\,)$, we then get that $v\,q_{V_0}(\mbi_{\dot\zb_n})$ and $q'_0\,w_0(\mbi_{\dot\zb_n})$ are homologous, but not necessarily equal. In $(\,$$\triangleleft_2$$\,)$, although $\zs_n\in\cB_{k-1}\setminus\cB_{k-2}$, its image $\zw_{k-1}(\zs_n)\in\cB'_{k-1}$ may also belong to $\cB'_{k-2}\,$. {Finally}, in $(\,$$\triangleleft_3$$\,)$, we find that  $vq_k(\mbi^k_{\zs_n})$ and $q'_k\zw_k(\mbi^k_{\zs_n})$ differ by a cycle, but do not necessarily coincide.\end{rem}

The next result describes cofibrant replacements.

\begin{theo}\label{FCRF} In $\tt DG\cD A$, the functorial cofibrant replacement functor $Q$, which is induced by the functorial `{\small Cof -- TrivFib}' factorization $(j,q)$ described in Theorem \ref{P:c-tf_tc-f}, is defined on objects $B\in\tt DG\cD A$ by $Q(B)=\cS V_B$, see Theorem \ref{P:c-tf_tc-f} and set $A=\cO$, and on morphisms $v\in{\tt DG\cD A}(B,B')$ by $Q(v)=\zw$, see Equations (\ref{w}), (\ref{w_kDef}), and (\ref{w_0}), and set $\zw_0=w_0$. Moreover, the differential graded $\cD$-algebra $\cS{\cal V}_B$, see Proposition \ref{NonFuncFact} and set $A=\cO$, is a cofibrant replacement of $B$. \end{theo}

\section{Acknowledgments}
The authors are grateful to Jim Stasheff for having read the first version of the present
paper. His comments and suggestions allowed to significantly improve their text.

\end{document}